\documentclass[11pt]{amsart}
\usepackage[utf8]{inputenc}
\usepackage{amsmath,amssymb,amsthm,mathrsfs,mathtools}
\usepackage{array,enumitem,yfonts}
\usepackage{comment}
\usepackage[normalem]{ulem}
\usepackage{bm}
\usepackage[alphabetic]{amsrefs}
\usepackage[all,cmtip]{xy}
\usepackage[colorlinks,anchorcolor=blue,citecolor=blue,linkcolor=blue,urlcolor=blue,bookmarksdepth=2]{hyperref}
\allowdisplaybreaks

\usepackage{fullpage}
\usepackage[margin=1in]{geometry}

\usepackage{cleveref}

\usepackage{tikz}
\usetikzlibrary{cd,arrows}
\tikzset{>=stealth}
\tikzcdset{arrow style=tikz}
\tikzset{link/.style={column sep=1.8cm,row sep=0.16cm}}

\AtBeginDocument{%
	\def\MR#1{}
}
\usepackage{fancyhdr}
\fancypagestyle{titlepage}
{
	\fancyhf{}

	\fancyfoot[l]{
		\href{https://mathscinet.ams.org/mathscinet/msc/msc2020.html}
		{\emph{2020 Mathematics Subject Classification}: 11B85, 13B40, 13B35}
		\\
		\emph{Keywords}: Automatic sequence; Christol's theorem; Henselization; beta expansion.
            \\
	}
}

\newcommand{\bN}{\mathbb{N}}    
\newcommand{\bZ}{\mathbb{Z}}    
\newcommand{\bF}{\mathbb{F}}    

\newcommand{\cC}{\mathcal{C}}   
\newcommand{\cD}{\mathcal{D}}   
\newcommand{\cI}{\mathcal{I}}   
\newcommand{\cS}{\mathcal{S}}   

\newcommand{\fa}{\mathfrak{a}}  
\newcommand{\fb}{\mathfrak{b}}  
\newcommand{\fc}{\mathfrak{c}}  
\newcommand{\fp}{\mathfrak{p}}	
\newcommand{\fm}{\mathfrak{m}}  
\newcommand{\fq}{\mathfrak{q}}  
\newcommand{\fn}{\mathfrak{n}}  

\newcommand{\bx}{\bm{x}}        

\DeclareMathOperator{\Spec}{Spec}	
\DeclareMathOperator{\id}{id}		
\DeclareMathOperator{\depth}{depth}     
\DeclareMathOperator{\height}{height}     
\DeclareMathOperator{\colimit}{colim}     

\newtheorem{Thm}{Theorem}[section]
\newtheorem{Prop}[Thm]{Proposition}
\newtheorem{Lem}[Thm]{Lemma}

\numberwithin{equation}{section}

\theoremstyle{definition}

\newtheorem{Defn}[Thm]{Definition}

\newtheorem{Rmk}[Thm]{Remark}

\begin{document}
	
  \title{\bf On the generalization of Christol's theorem to domains of finite type over finite fields}
  \author{Chunlin~Wang, Lian~Duan}
  \date{}
  \newcommand{\ContactInfo}{{
    \bigskip\footnotesize
    \bigskip 
        \noindent Chunlin~Wang,
    \textsc{School of Mathematical sciences\\
	      Sichuan Normal University\\
		  Jinjiang, Chengdu, China}\par\nopagebreak
		  \noindent\texttt{c-l.wang@outlook.com} 
          
    \noindent Lian~Duan,
    \textsc{Institute of Mathematics science\\
	      ShanghaiTech University\\
		  Pudong, Shanghai, China}\par\nopagebreak
		  \noindent\texttt{duanlian@shanghaitech.edu.cn}
    }}
	
  \begin{abstract} 

We generalize Christol's theorem---which characterizes algebraic power series over a finite field in terms of automatic coefficient sequences---to completions of domains of finite type over a finite field~\(\mathbb{F}_{q}\). Let \(A\) be such a domain of dimension~\(d\) and \(\mathfrak{a}\) an ideal of depth~\(d\). In the \(\mathfrak{a}\)-adic completion \(\hat{A}\), choose an \(\hat{A}\)-regular sequence \(\boldsymbol{x}=\{x_{1},\dots,x_{d}\}\subset\hat{\mathfrak{a}}\) and a set of representatives \(E\) of \(\hat{A}/(\boldsymbol{x})\) that is an \(\mathbb{F}_{q}\)-vector subspace consisting of algebraic elements. Every element \(y\in\hat{A}\) then admits a unique expansion \(y=\sum_{u\in\mathbb{N}^{d}}a_{u}x^{u}\) with coefficients \(a_{u}\in E\). We prove that \(y\) is algebraic over \(A\) if and only if the coefficient sequence \(\{a_{u}\}_{u\in\mathbb{N}^{d}}\) is \(q\)-automatic. The proof relies on a careful study of henselizations, which are shown to coincide with the algebraic closures inside the completions under certain conditions. As an application, we obtain a complete characterization of the automaticity of \(\beta\)-expansions for formal Laurent series over \(\mathbb{F}_{q}\), answering questions posed by Hbaib--Mkaouar and Scheicher--Sirvent.
  \end{abstract}

  \maketitle
    \thispagestyle{titlepage}
	
    \tableofcontents

    \newpage
    
\section{Introduction}\label{Sect: intro} 
Let $\bF_q$ be a finite field of $q$ elements. A power series $f(x)\in \bF_q[[x]]$ is called \emph{algebraic} (over $\bF_q[x]$) if there exists a two-variable polynomial $P(x,y)$ over $\bF_q$ such that $P(x,f(x))=0$. A famous result on algebraic power series, due to Christol, states that a power series $f(x)=\sum_{i\ge 0}a_ix^i$ is algebraic if and only if its sequence of coefficients $\{a_i\}_{i\in \bN}$ is $q$-automatic. Here a sequence $\{a_i\}_{i\in \bN}$ is called (one-dimensional) \emph{$q$-automatic} if the following set of subsequences
$$\{ \{a_{i q^j+k}\}_{i\in \bN}: j\ge 0, 0\le k<q^j \}$$ 
is finite (see \Cref{Defn: automatic_sequence} and \Cref{Lem: k-kernel} for more information on automatic sequences). 

Christol's theorem stands as a cornerstone in positive characteristic arithmetic because it establishes a profound dictionary between algebra and computation. It perfectly equates a purely algebraic property—being an algebraic power series over $\mathbb{F}_q$—with a combinatorial one: having a $q$-automatic coefficient sequence. This equivalence provides a powerful bridge that allows mathematicians to decode the intricate structure of algebraic series by analyzing the deterministic automata that generate their coefficients. Consequently, daunting questions surrounding transcendence and algebraicity are reduced to tractable, discrete problems, such as bounding the $q$-kernel. A striking testament to the power of this approach is its successful application in proving the transcendence of the Carlitz period (see for example \cite{Allouche}). 

Christol's theorem is fascinating also because it provides a contrast to the behavior of real numbers, as highlighted by the normal number conjecture. For an integer $b \ge 2$, a real number $x$ is said to be normal in base $b$ if, in its base-$b$ expansion, every finite block of digits of length $k$ appears with an asymptotic frequency of $1/b^k$. A number is absolutely normal if it is normal in every integer base $b \ge 2$. Borel's normal number conjecture \cite{Borel} states that every irrational algebraic real number is absolutely normal. Since a normal number is never automatic, this conjecture implies that any irrational number with an automatic base-$b$ expansion must be transcendental. Actually, although the normal number conjecture remains open, a partial result towards this conjecture by Adamczewski and Bugeaud \cite{Adamczewski-Bugeaud} already implies that irrational real numbers with automatic base-$b$ expansion is transcendental. However, the situation is entirely different in positive characteristic: Christol's theorem ensures that every algebraic element over $\mathbb{F}_q(x)$ possesses an automatic coefficient sequence.

Due to its importance, the study of Christol's theorem has attracted much attention. Its generalizations and analogues are extensively considered in literature. Salon \cites{Salon87,Salon89} generalized Christol's Theorem to multivariable power series over finite fields. Independently, Denef and Lipshitz \cite{Denef-Lipshitz}, Sharif and Woodcock \cite{Sharif-Woodcock}, as well as Harase \cites{Harase88, Harase89} generalized Christol's Theorem to the case of multivariable power series over any field of positive characteristic. An extension of Christol's theorem to generalized algebraic power series is given by Kedlaya \cite{Kedelaya}. The automaticity of $\beta$-expansion of formal Laurent series is considered in \cite{Hbaib-Mkaouar} and \cite{Scheicher-Sirvent}. Recently, an arithmetic analogue of Christol's theorem for Witt vectors over algebraic number fields are acquired \cite{Uramoto}.  {These generalizations and analogues of Christol's theorem underscores its role as a cornerstone theorem in characteristic $p$ algebra, linking the theory of algebraic series directly to finite automata theory and offering concrete combinatorial invariants for algebraic objects.}

In general, whenever an algebraic structure admits ``expansion'' of its elements, it is natural to ask whether a Christol-type theorem holds---that is, whether a kind of algebraicity corresponds to a certain automaticity of expansion coefficients. From this perspective, we consider a generalization of Christol's theorem for certain power series expansions in completions (with respect to given ideals) of domains of finite type over finite fields. Such expansions, which may be treated as a generalization of classical power series over finite fields, frequently arise in commutative algebra. 

To state our results, we set up the following framework. The reader is referred to \Cref{Sect: power_ser_exp} for more background about the related terminologies and facts. Let $A$ be an integral domain of finite type over $\bF_q$ of Krull dimension $d$. Let $\fa$ be an ideal in $A$ of depth $d$. Let $\hat{A}$ be the completion of $A$ at $\fa$. Then $\hat{\fa}:=\fa\hat{A}$ is an ideal of depth $d$ in $\hat{A}$. Fix an $\hat{A}$-regular sequence $\bx:=\{x_1,\dots,x_d\}$ of length $d$ in $\hat{\fa}$. Let $(\bx)$ be the ideal of $\hat{A}$ generated by $x_1,\dots,x_d$. Then $\hat{A}/(\bx)$ is finite. Let $E$ be a \emph{set of representatives} of $\hat{A}/\hat{\bx}$, i.e., $E$ is a subset of $\hat{A}$ such that the natural composition 
$$E\hookrightarrow \hat{A} \to \hat{A}/(\bx)$$ 
is a bijection of sets, where the first map is the inclusion and the second is the canonical quotient map. We show certain power series expansions with respect to $\bx$ and $E$ exist for elements of $\hat{A}$. 

\begin{Thm}\label{Thm: exist_uniq_power_series}
Let $\bx=\{x_1,\dots,x_d\}$ and $E$ be described as above. Then every $y\in \hat{A}$ can be uniquely written as 
$$
y=\sum_{u\in \bN^d} a_ux^u, \quad  a_u\in E ~ {\rm for ~ all ~} u\in \bN^d,
$$
where for $u=(u_1,\dots,u_d)\in\bN^d$, $x^u=x_1^{u_1} \cdots x_d^{u_d}$.
\end{Thm}

If we take $A=\bF_q[x_1, \dots ,x_d]$ to be the polynomial ring over $\bF_q$, $\fa=(x_1, \dots, x_d)$ to be the ideal generated by $x_1, \dots, x_d$, $\bx=\{x_1, \dots, x_d\}$ and $E=\bF_q$, then $\hat{A}=\bF_q[[x_1, \dots, x_d]]$ and the expansions in \Cref{Thm: exist_uniq_power_series} recover the classical power series over $\bF_q$. The main purpose of this paper is to generalize Christol's theorem from the power series rings to the setup in Theorem \ref{Thm: exist_uniq_power_series}. 

Recall that an element $y\in \hat{A}$ is \emph{algebraic} (over $A$) if there exists $P(T)\in A[T]$ such that $P(y)=0$. For a given positive integer $k$, like the classical case, a sequence $\{a_u\}_{u\in \bN^d}$ over a finite set is called ($d$-dimensional) \emph{$k$-automatic} if the following set of subsequences
$$\{ \{a_{u k^i+v}\}_{u\in \bN}: i\ge 0, v=(v_1, \dots, v_d), 0\le v_j<k^i \}$$ 
is finite. With these notations, we show that 

\begin{Thm}\label{Thm: algebraic_eqv_automatic}
With the notation of Theorem \ref{Thm: exist_uniq_power_series}. Assume elements of $\bx$ and $E$ are algebraic over $A$. If $E$ is an $\bF_q$-vector subspace of $\hat{A}$, then $y$ is algebraic if and only if $\{a_u\}_{u\in \bN^d}$ is $q$-automatic. 
\end{Thm}

\begin{Rmk}
In \Cref{Thm: algebraic_eqv_automatic}, the condition ``$E$ is an $\bF_q$-vector subspace of $\hat{A}$" can be equivalently replaced by ``$E$ is an additive subgroup of $\hat{A}$". To see this, one direction is clear. As for the converse statement, first notice that every additive subgroup of $A$ is automatically an $\bF_p$-vector subspace, and then use the well-known fact that every $p$-automatic sequence is also a $q$-automatic sequence, see for example \cite{AS-auto-seqn}*{Theorem~6.6.4}. Since we employ the tools from commutative algebra in our proof, the above additive group restriction on $E$ is considered natural.
\end{Rmk}

Since we study sequence over a group $E$ other than a ring or field, the traditional reliance on Cartier operators or diagonals of multivariable rational functions are no longer applicable. Instead, a machinery from commutative algebra: the henselization of a pair, emerges in a natural way. We demonstrate that the algebraic closure $A^a$ of $A$ inside the completion $\hat{A}$ coincides precisely with the henselization $(A, \mathfrak{a})^h$ under the given depth conditions. By showing that $A^h$ is a free $B^h$-module (where $B$ is the subring generated by the regular sequence), we reduce the problem to the classical case over power series rings. Our technical approach indicate a link between Christol’s theorem and the foundational principles of commutative algebra.

The motivation for developing this broad framework is not merely structural, but driven by concrete open problems. A prime example is the study of $\beta$-expansions for formal Laurent series over $\bF_q$. 
The theory of $\beta$-expansions originates in the seminal work of R\'enyi, who introduced representations of real numbers in non-integer bases as a fundamental tool in dynamical systems and number theory. Since then, $\beta$-expansions have played important roles in the study of numeration systems, symbolic dynamics, and the arithmetic complexity of digit sequences. In positive characteristic, an analogous theory arises naturally for formal Laurent series over finite fields. Motivated by Christol’s theorem, the automaticity of these expansions has been a subject of intense study, yet a complete characterization remained elusive, leaving open questions posed such as in \cite{Hbaib-Mkaouar} and \cite{Scheicher-Sirvent}. Because the algebraic structures underlying these $\beta$-expansions are not comfortably handled by classical power series techniques, they serve as an ideal touchstone for our theory.

To be precise, we connect $\beta$-expansions to a special case of expansions in \Cref{Thm: exist_uniq_power_series}. Then the equivalence of automaticity and algebraicity of $\beta$-expansion becomes a special case of \Cref{Thm: algebraic_eqv_automatic}. In particular, our main theorem yields a complete characterization of automatic $\beta$-expansions for formal Laurent series over finite fields, thereby answering the questions raised in the works \cite{Hbaib-Mkaouar} and \cite{Scheicher-Sirvent}.

Beyond the applications established in this paper, we hope that the use of henselization in our approach may also provide a useful perspective for extending Christol-type theorems to rings that are not integral domains. Let $\bZ_p$ be the ring of $p$-adic integers, and $\sum_{u\in \bN^d} a_u x^u$ a power series over $\bZ_p$ of $d$ variables. Denef and Lipshitz proved that the reduction $\{a_n \mod p^s\}_{n\ge 0}$ is an automatic sequence over $\bZ_p/p^s\bZ_p$, where $s$ is an positive integer, if and only if $\sum_{u\in \bN^d} a_u x^u$ is congruent to an algebraic power series over $\bZ_p[x]$ modulo $p^s$. This result suggest a potential generalization of Christol's theorem to rings that are not domain (Denef and Lipshitz's result actually involves power series over $\bZ_p/p^{\alpha}\bZ_p$, which is not a domain). But in these rings, algebraic elements is not well defined, we believe henselization is a good substitution for the concept of algebraic elements.  

The paper is organized as follows. In \Cref{Sect: power_ser_exp} we review the relevant background on regular sequences and completions, which allows us to state and prove a refined version of \Cref{Thm: exist_uniq_power_series}. \Cref{Sect: automaticity}  introduces the notion of automatic sequence, and explains how \Cref{Thm: algebraic_eqv_automatic} is reduced to a key statement about algebraic closures (\Cref{Thm: free_mod_alg_version}). The proof of \Cref{Thm: free_mod_alg_version} occupies \Cref{Sect:  Henselization}, where we recall the henselization of a pair, show that under proper conditions, the henselization coincides with the algebraic closure inside the completion, and then establish the required module‑theoretic property using properties of henselian rings. Finally, we discuss some examples concerning \Cref{Thm: algebraic_eqv_automatic}, and prove a theorem which answers questions posed in \cite{Hbaib-Mkaouar} and \cite{Scheicher-Sirvent} about $\beta$-expansions for formal Laurent series over $\bF_q$. 

\section{Power series expansions}\label{Sect: power_ser_exp}
In this section, we first recall some background from commutative algebra. In particular, the facts about regular sequences, completion with respect to an ideal and the related power series expansion will be reviewed. After that, a refined version of Theorem \ref{Thm: exist_uniq_power_series} is proved. Throughout this paper, all rings are assumed to be commutative with identity, and all ideals are assumed to be proper. 

\subsection{Regular sequence and power series}
Let $A$ be a ring, the \emph{dimension} of $A$, written as $\dim A$, is the Krull dimension, i.e., the supremum of lengths of chains of prime ideals in $A$. Let $\fp$ be a prime ideal of $A$, the {\it height} of $\fp$, written $\height \fp$, is defined as the supremum of lengths of chains of prime ideals descending from $\fp$. For a general ideal $\fa$ of $A$ the height of $\fa$ is defined to be the minimum of heights of all prime ideals containing $\fa$, i.e., $\height \fa=\min\{\height \fp| \fp\supset \fa\}$. 

\begin{Defn}\label{Defn: reg_seq}
Let $M$ be an $A$-module. A sequence of elements $\{x_1,\dots, x_n\}$ of $A$ is said to be {\it $M$-regular} if 
\begin{itemize}
    \item[(i)] $M/(x_1,\dots,x_n)M\neq 0$; 
    \item[(ii)] for each $1\le i\le n$, $x_i$ is not a zero-divisor of $M/(x_1,\dots,x_{i-1})M$.  
\end{itemize}
When $M=A$, an $A$-regular sequence would simply be called a regular sequence. Moreover, if every element $x_i$ is contained in an ideal $\fa$, we say $\{x_1, \dots, x_n\}$ is an ($M$-)regular sequence in $\fa$. 
\end{Defn} 

Suppose $A$ is a Noetherian ring, and $M$ is a finite $A$-module. For an ideal $\fa$ in $A$ with $\fa M\neq M$, the {\it depth} of $\fa$ on $M$, written $\depth(\fa,M)$, is defined to be the supremum of lengths of $M$-regular sequences in $\fa$. In particular, for an ideal $\fa$, $\depth \fa:=\depth(\fa,A)$ is called the \emph{depth} of $\fa$. 

\begin{Lem}\label{Lem: properties_on_reg_seq} 
Let $A$ be a Noetherian ring, and $\fa$ an ideal of $A$. The following are true: 
\begin{enumerate}
    \item The depth of $\fa$ is less than or equal to its height;
    \item $\fa$ shares the same depth with its radical;
    \item Given an $A$-regular sequence $\bx$, and a flat homomorphism $\varphi: A\rightarrow B$. Let $\fb$ be the ideal generated by elements of $\varphi(\bx)$.  If $\fb\neq B$, then $\varphi(\bx)$ is a $B$-regular sequence. 
    \item $\depth(\fa, A)=\inf\limits_{\fp \supset \fa } \{\depth(\fp A_{\fp}, A_{\fp})\}$, where $\fp$ runs through all prime ideals in $A$ containing $\fa$.  
\end{enumerate}
\end{Lem}

\begin{proof}
(1) See \cite{Eisenbud}*{Proposition 18.2}. (2) See \cite{Eisenbud}*{Corollary 17.8}. (3) Take $R=M=A$, $S=N=B$ in \cite{Cohen-Macaulay-rings}*{Proposition 1.2}. (4) See \cite{Matsumura-commutative-alg}*{(15.F), Proposition}.
\end{proof}

According to this lemma, if $A$ is local with maximal ideal $\fm$, then $\depth \fa\leq {\rm height\ }\fm=\dim A$.

\begin{Defn}\label{Defn: Cohen-Macaulay}
    Let $(A, \fm)$ be a Noetherian local ring, then it is called \emph{Cohen-Macaulay} if $\depth \fm=\dim A$. A general Noetherian ring $A$ is \emph{Cohen-Macaulay} if for every maximal ideal $\fm$, the corresponding localization $A_\fm$ is Cohen-Macaulay.  
\end{Defn}

We collect some useful properties of Cohen-Macaulay rings for later use. 
\begin{Prop}\label{Prop: CM_basic_prop}
    Let $(A, \fm)$ be a local Noetherian ring. 
    \begin{enumerate}
        \item If $A$ is regular, then it is Cohen-Macaulay. 
        \item If $A$ is regular, then any localization of $A$ at a prime ideal is also Cohen-Macaulay. 
        \item If $A$ is Cohen-Macaulay, then a set of elements $\{x_1, \dots, x_r\}$ is a regular sequence if and only if $\dim A/(x_1, \dots, x_r)=\dim A-r$. 
        \item If $A$ is Cohen-Macaulay, and $\{x_1, \dots, x_r\}$ is a regular sequence, then $A/(x_1, \dots, x_r)$ is also Cohen-Macaulay. 
        \item Let $\hat{A}$ be $\fm$-adic completion of $A$. Then $A$ is Cohen-Macaulay if and only if $\hat{A}$ is Cohen-Macaulay.
    \end{enumerate}
\end{Prop}
\begin{proof}
    For (1)-(4), see for example \cite{Hartshorne}*{II, Theorem~8.21A}, or \cite{Matsumura-commutative-alg}*{(1). p. 121; (2) p.104; (3) p.105; (4) p. 104}. For (5), see \cite{Eisenbud}*{Proposition~18.8}.  
\end{proof}

For an ideal $\fa$ in $A$, the {\it associated graded ring of $A$ with respect to $\fa$} is defined by 
$$G_{\fa}(A):=\bigoplus_{i=0}^{\infty} \fa^i/\fa^{i+1}.$$ 
The multiplication in $G_{\fa}(A)$ is induced by the multiplication $\fa^i\times \fa^j\rightarrow \fa^{i+j}$. Similarly, we define the {\it associated graded module}  
$$G_{\fa}(M):=\bigoplus_{i=0}^{\infty}\fa^iM/\fa^{i+1}M.$$
One can check that $G_{\fa}(M)$ is a graded $G_{\fa}(A)$-module in the natural sense. 

Let $X_1,\dots, X_n$ be indeterminants and $A[X]\coloneqq A[X_1,\dots,X_n]$ the polynomial ring over $A$. The {\it polynomial ring over $M$} is defined to be the module 
$$M[X]:=M \otimes_A A[X].$$ 

Suppose $\fa$ is generated by $x_1,\dots,x_n$.  Then there is a graded ring homomorphism from $A[X]$ to $G_{\fa}(A)$ induced by the rule 
$$aX_1^{u_1}\cdots X_n^{u_n} \mapsto ax_1^{u_1}\cdots x_n^{u_n} \mod \fa^{|u|+1},$$ 
where for $u=(u_1,\dots,u_n)\in \bN^{n}$, set $|u|\coloneqq u_1+\cdots+u_n$. This homomorphism is surjective, and contains $\fa[X]$ in the kernel. Hence it induces a graded ring epimorphism from $A/\fa [X]$ to $ G_{\fa}(A)$.  Likewise there is a canonical epimorphism 
\begin{equation}\label{Eqn: poly_ring_to_grad_ring}
(M/\fa M) [X] \rightarrow G_{\fa}(M).
\end{equation} 
defined by 
$$\bar{m}\otimes X_1^{u_1}\cdots X_n^{u_n} \mapsto mx_1^{u_1}\cdots x_n^{u_n} \mod \fa^{|u|+1} M,$$ 
where for $\bar{m}\in M/\fa M$, $m$ is any element of $M$ whose reduction modulo $\fa M$ equals $\bar{m}$. 

\begin{Prop}\cite{Cohen-Macaulay-rings}*{ Theorem~1.1.8}\label{Prop: M_seq_isom_graded_mod}
Let $A$ be a ring, $M$ an $A$-module, $\bm{x}=\{x_1, \cdots, x_n\}$ an $M$-regular sequence, and $\fa=(\bm{x})=(x_1, \cdots, x_n)$. Then the induced map (\ref{Eqn: poly_ring_to_grad_ring}) is an isomorphism. In particular, if $\bx$ is $A$-regular, then $(A/\fa) [X]$ is isomorphic to $G_{\fa}(A)$. 
\end{Prop}

\begin{Defn}
    Let $\bx=\{x_1,\dots, x_n\}$ be a sequence of elements in $A$. Let $\fa$ be the ideal generated by $\bx$. We say $\bx$ is {\it $M$-quasi-regular} if the induced canonical homomorphism (\ref{Eqn: poly_ring_to_grad_ring}) is an isomorphism.
\end{Defn}
\Cref{Prop: M_seq_isom_graded_mod} says that an $M$-regular sequence is $M$-quasi-regular. The converse holds under certain conditions.

\begin{Lem}\label{Lem: quasi-reg_eqv_reg}
Let $A$ be a Noetherian ring, and $\bx=\{x_1,\dots,x_n\}$ a sequence of elements in the Jacobson radical of $A$. 
Then $\bx$ is $A$-regular if and only if $\bx$ is $A$-quasi-regular.
\end{Lem}  

\begin{proof}
The implication $A$-regular $\Rightarrow$ $A$-quasi-regular is given by Lemma \ref{Prop: M_seq_isom_graded_mod}; for the converse, see \cite{Matsumura-commutative-ring-theory-1986}*{Theorem 16.3} and the remark below it. 
\end{proof}

Given a ring $A$ and an ideal $\fa$, the {\it $\fa$-adic completion of $A$}, or the {\it completion of $A$ at $\fa$}, is defined as the inverse limit
$$\hat{A}:=\varprojlim_{n} A/\fa^n A.$$ 
Denote by $\hat{\fa}:=\fa\hat{A}$ the $\fa$-adic completion of $\fa$. There is a canonical homomorphism of rings $A\rightarrow \hat{A}$ whose kernel is $\cap_{n=1}^{\infty} \fa^n$. Thus in the case $\cap_{n=1}^{\infty} \fa^n=(0)$ we consider $A$ as a subring of $\hat{A}$. If furthermore the canonical homomorphism $A \rightarrow \hat{A}$ is an isomorphism, then we say $A$ is {\it $\fa$-adically complete}. Notice that if $A$ is $\fa$-adically complete, then $A$ is Hausdorff with respect to the $\fa$-adic topology on $A$. 

Suppose $A$ is $\fa$-adically complete. For a sequence of elements $\{x_1,\dots, x_n\}$ in $\fa$, let $I$ be the ideal generated by $x_1,\dots,x_n$, and $E\subset A$ a set of representatives of $A/I$. Consider the formal series  
$$\sum_{u\in \bN^n} a_ux^u, \quad  \forall u\in\bN^d, a_u\in E.$$ 
For all integers $i\ge 0$, let  
$$y_i=\sum_{|u|\le i}a_ux^u$$ 
be the partial sums of the above series. Then $y_i$ converges to a unique element $y\in A$ with respect to the $\fa$-adic topology of $A$. We call $y=\sum_{u\in \bN^n} a_ux^u$ a {\it power series expansion of $y$ with respect to $\bx$ and $E$}. We show that the power series expansions exist. 

\begin{Lem}\label{Lem: exist_power_series} 
Let $A$ be an $\fa$-adically complete ring. Let $\bx$ be a sequence of elements in $\fa$, $I=(\bx)$ the ideal generated by $\bx$, and $E\subset A$ a set of representatives of $A/I$. 
Then every element $y\in A$ admits a power series expansion with respect to $\bx$ and $E$. 
\end{Lem}

\begin{proof}
It is enough to show that for all $i\ge 0$, there are $y_i=\sum_{|u|\le i}a_ux^u$ such that $y-y_i\in I^{i+1}$. We do this by induction on $i$. For $i=0$, since $E$ is a set of representatives of $A/I$, there exists $a_0\in E$ such that $y-a_0\in I$. Now for $i>0$, suppose we have found $y_{i-1}$. Since $y-y_{i-1}\in I^{i}$. And since $I^{i}$ is generated by $\{x^u: |u|=i\}$. One finds $y-y_{i-1}=\sum_{|u|=i}b_ux^u$ with $b_u\in {A}$ and $|u|=i+1$. Now choose $a_u \in E$ satisfies $a_u-b_u \in I$, and take $y_i=y_{i-1}+\sum_{|u|=i}a_ux^u$. Direct calculation verifies that $y-y_i\in I^{i+1}$. This completes the proof.  
\end{proof}

For an arbitrary sequence $\bx\subset \fa$ and a set of representatives $E$ of $\hat{A}/(\bx)$, there may exist different $\{a_u\}_{u\in\bN^d}$ and $\{b_u\}_{u\in\bN^d}$ such that $a_u, b_u\in E$ for all $u$, and
$$y=\sum_{u\in \bN^d}a_u x^u=\sum_{u\in \bN^d}b_u x^u.$$
However, the uniqueness of the power series expansion is fulfilled when $\bx$ is $A$-quasi-regular.  

\begin{Prop}\label{Prop: uniq_expan_eqv_quasi-reg}
With the notations and assumptions in \Cref{Lem: exist_power_series}, the following are equivalent: 

(1) for every $y\in A$, there exists a unique power series expansion of $y$ with respect to $\bx$ and $E$;  

(2) $\bx$ is $A$-quasi-regular. \\
If furthermore $A$ is Noetherian, then (1) and (2) are equivalent to 

(3) $\bx$ is $A$-regular. 
\end{Prop} 

\begin{proof}
(1)$\Rightarrow$(2). Suppose every $y\in A$ has a unique power series expansion of $y$ with respect to $\bx$ and $E$. We show $\bx$ is $A$-quasi-regular. To the contrary suppose $\bx$ was not $A$-quasi-regular. Then the map $\varphi: (A/I)[X] \rightarrow G_I(A)$ is an epimorphism but not an isomorphism. Write 
$$(A/I)[X]=\bigoplus_{i=0}^\infty M_i,$$ 
where $M_i$ denotes the degree $i$ homogeneous component of $(A/I)[X]$. Noticing that $\varphi$ is a homogeneous homomorphism of graded rings, $\varphi$ maps $M_i$ onto the homogeneous component $I^i/I^{i+1}$ of $G_I(A)$. Since $\varphi$ is not an isomorphism, there exists an integer $i\ge 0$ such that $\varphi|_{M_i}: M_i\rightarrow I^i/I^{i+1}$ is not an isomorphism. So there exists a nonzero homogeneous polynomial $\sum_{|u|=i}\bar{b}_uX^u \in A/I[X]$ satisfying 
$$\varphi\big(\sum_{|u|=i}\bar{b}_uX^u \big)=\sum_{|u|=i}\bar{b}_u\bar{x}^{u}=0 \in I^i/I^{i+1}.$$ 
For fixed $y\in A$, recalling the inductive procedure in the proof of Lemma \ref{Lem: exist_power_series}, suppose we have found a partial sum  
$$y_{i}=\sum_{|u|\le i} a_ux^u, a_u\in E$$
with $y-y_i\in I^{i+1}$. Meanwhile for each $u$ with $|u|=i$, there exists a $c_u\in E$ such that the reduction of $c_u-a_u$ modulo $I$ equals $\bar{b}_u$. Notice that $c_u\neq a_u$ if $\bar{b}_u\neq 0$. Let 
$$y_i'=\sum_{|u|< i} a_ux^u+\sum_{|u|=i}c_ux^u.$$
Then $y_i\equiv y_i' \mod I^i$. So both $y-y_i$ and $y-y_i'$ are in $I^{i+1}$.  Again by the inductive procedure in the proof of Lemma \ref{Lem: exist_power_series}, we can find power series expansions for both $y-y_i$ and $y-y_{i}'$. Expansions $y-y_i$ and $y-y_i'$ add the partial sums $y_i$ or $y_i'$ respectively give different power series expansions of $y$ with respect to $\bx$ and $E$. This contradicts our uniqueness assumption on power series of $y$. Therefore $\bx$ is $A$-qausi-regular. 

(2)$\Rightarrow$(1). Next assume $\bx$ is $A$-quasi-regular. We show there exists a unique power series expansion of $y$ for all $y\in A$. Given $y\in A$, let $\sum_{u\in\bN^d} a_ux^u$ and $\sum_{u\in\bN^d}b_ux^u$ be two power series expansions of $y$ with respect to $\bx$ and $E$. We show $a_u-b_u=0$ for all $u\in\bN^d$. This will be done by induction on $|u|$. 

If $|u|=0$, then $a_0$ (or $b_0$) is the element of $E$ satisfying $y\equiv a_0 \mod I$ (or $y\equiv b_0 \mod I$). So $a_0=b_0$. Suppose $a_u=b_u$ for all $u$ with $|u|<i$. Then 
$$\sum_{|u|=i}a_ux^u-\sum_{|u|=i}b_ux^u\in I^{i+1}.$$
Let $\psi: A[X] \rightarrow G_I(A)$ be the epimorphism induced by $\bx$. Recall that $\varphi(\alpha X^u)=\bar{\alpha}x^u$ for $\alpha\in A$. We deduce that  
$$\varphi(\sum_{|u|=i}a_uX^u)=\varphi(\sum_{|u|=i}b_uX^u).$$
That is $\sum_{|u|=i}(a_u-b_u)X^u\in \ker \psi$. Since $\bx$ is $A$-quasi-regular, $\ker \psi=I[X]$. Thus $a_u-b_u\in I$ for all $u$ with $|u|=i$. But both $a_u$ and $b_u$ are elements of $E$. Hence $a_u=b_u$. This shows the uniqueness of power series expansions of $y$ with respect to $\bx$ and $E$. (1)$\Leftrightarrow$(2) is proved. 

Further assume $A$ is Noetherian. Since $A$ is $\fa$-adically complete, $\fa$ is contained in the Jacobson radical of $A$. So by Lemma \ref{Lem: quasi-reg_eqv_reg}, (2) $\Leftrightarrow$ (3). This finishes the proof of the lemma. 
\end{proof}

\subsection{Proof of Theorem \ref{Thm: exist_uniq_power_series}} In  this subsection, we will state and prove a refined version of \Cref{Thm: exist_uniq_power_series}. Still take $A$ to be an integral domain of finite type over $\bF_q$ of dimension $d$, with an ideal  $\fa$. The following fact is standard in commutative algebra, and a proof is given for the convenience of the reader. 

\begin{Lem}\label{Lem: height_quot_decomp}
Then the following are equivalent. 

\begin{itemize}
\item[(1)] $\height \fa=\dim A$; 
\item[(2)] $A/\fa$ is a finite ring; 
\item[(3)] $\fa=\fq_1\cdots\fq_s$, where $\fq_i$ are primary ideals whose radicals are different maximal ideals.
\end{itemize}
\end{Lem} 

\begin{proof}
(1)$\Leftrightarrow$(2). Since $\dim A=\height \fa+ \dim A/\fa$, one sees that $\height \fa=\dim A$ if and only if $\dim A/\fa=0$. The latter happens exactly when $A/\fa$ is Artinian. This is equivalent to saying $A/\fa$ is finitely generated as an $\bF_q$-module, in other words, $A/\fa$ is a finite ring. 

(2)$\Leftrightarrow$(3). Let  $\fa=\fq_1\cdots\fq_s$ be a primary decomposition of $\fa$ such that the radicals of $\fq_i$ are different prime ideals. Then $\height \fa=\dim A$ if and only if the minimum of heights of radicals of $\fq_i$ is $\dim A$, if and only if all radicals of $\fq_i$ are maximal.   
\end{proof}

Let us recall the setups in \Cref{Sect: intro} with more details. Let $\hat{A}$ and $\hat{\fa}$ be $\fa$-adic completions of $A$ and $\fa$ respectively. Then $\hat{A}$ is neotherian, $\hat{\fa}$ is contained in the Jacobson radical of $\hat{A}$, and the canonical map $A\rightarrow \hat{A}$ is injective and flat (\cite{Atiyah-Macdonald}*{Proposition~10.14, 10.15(d) and Corollary~10.18}). Identify $A$ with its image in $\hat{A}$, for a sequence of elements $\bx=\{x_1,\dots,x_n\}$ in $\hat{\fa}$, let $E\subset \hat{A}$ be a set of representatives of $\hat{A}/(\bx)$. For our generalization of Christol's theorem---in particular, to define $q$-automatic sequences over $E$ in \Cref{Sect: automaticity}---we require $E$ to be finite. Thus $\hat{A}/I$ must be a finite ring, where $I=(\bx)$. So $A/\fa\cong \hat{A}/\hat{\fa}\cong (\hat{A}/I)/(\fa/I)$ are finite and $\height \fa=d$ due to \Cref{Lem: height_quot_decomp}. For this reason, we additionally assume $\height \fa=d$. 

\begin{Lem}\label{Lem: dim_height_depth_in_completion}
Under the assumption $\height \fa=d$, we have 
\begin{enumerate}
    \item $\dim \hat{A}=\height \hat{\fa}=d$; 
    \item every $A$-regular sequence in $\fa$, seen as a sequence in $\hat{\fa}$, is $\hat{A}$-regular. 
    \item $\depth \fa=d$ if and only if $\depth \hat{\fa}=d$. 
\end{enumerate}
\end{Lem}

\begin{proof}
(1) Let $\fa=\prod_{i=1}^s \fq_i$ be the primary decomposition of $\fa$ as stated in \Cref{Lem: height_quot_decomp}. Let $\fm_i$ be the radical of $\fq_i$. Let $\hat{A}_i$ be the completion of $A_i$ at $\fm_iA_i$. 
Then 
$$\hat{A}=\varprojlim_{n}A/\fa^n \cong \varprojlim_{i} \prod_{i=1}^s A/\fq_i^n 
\cong \prod_{i=1}^s \varprojlim_{n} A/\fq_i^n \cong \prod_{i=1}^s \varprojlim A/\fm_i^n 
\cong \prod_{i=1}^s \hat{A}_i.$$
Since $\dim A=d$, we have $\dim \hat{A}_i=d$ for all $1\le i\le s$. It follows that $\dim \hat{A}=d$. 

By \Cref{Lem: height_quot_decomp}, $\height \fa=d$ implies that $\hat{A}/\hat{\fa}$ ($\cong A/\fa$) is a finite ring. So all prime ideals containing $\hat{\fa}$ are maximal. Meanwhile all maximal ideals of $\hat{A}$ are of the form $\fm_i\hat{A}$. We have $\height \fm_i\hat{A}=\height \fm_i\hat{A}_i=d$. That is, heights of all maximal ideals of $\hat{A}$ are $d$. Thus $\height \hat{\fa}=d$. 

(2) Note that $\hat{A}$ is a flat $A$-algebra. Then (2) follows from \Cref{Lem: properties_on_reg_seq} (3). 

(3) Obviously (2) implies $\depth \fa\le \depth \hat{\fa}$. This together with $\depth \hat{\fa}\le \height \hat{\fa}=d$ deduce $\depth \fa=d \Rightarrow \depth \hat{\fa}=d$. Conversely, suppose $\depth \hat{\fa}=d$, this is equivalent to say $\hat{A}$ is Cohen-Macaulay. Then $\hat{A}_i$, which coincides with the localization of $\hat{A}$ at $\fm_i \hat{A}$, is also Cohen-Macaulay for any $1\le i\le r$. By \Cref{Prop: CM_basic_prop} (5), $A_i$ are Cohen-Macaulay for all $1\le i\le r$, i.e., $\depth (\fm_i A_i, A_i)=d$. Now $\depth \fa=d$ follows from \Cref{Lem: properties_on_reg_seq} (4).  
\end{proof}

\begin{Lem}\label{Lem: max_reg_seq_eqv_fin_quot}
Let $\bx=\{x_1,\dots, x_n\}$ be a $\hat{A}$-regular sequence in $\hat{a}$, and $I=(\bx)$ the ideal generated by $\bx$. Then $\hat{A}/I$ is a finite ring if and only if $n=d$.  
\end{Lem}

\begin{proof} 
Obviously $n\le \depth \hat{\fa} \le \height \hat{\fa}=d$. Suppose $\hat{A}/I$ is a finite ring. Then every prime ideal containing $I$ is maximal. So by the same arguments in the proof of \Cref{Lem: dim_height_depth_in_completion}, we deduce that $\height I=d$. Meanwhile, $\height I\le n$ by \cite{Atiyah-Macdonald}*{Corollary 11.16}. So $d\le n$. Now we conclude that $n=d$. 

Conversely suppose $n=d$. For an integer $m>0$, $\hat{A}/\hat{\fa}^m\cong A/\fa^m$. By the fact that $\height \fa^m=\height \fa=d$ and \Cref{Lem: height_quot_decomp}, we derive that $A/\fa^m$ is a finite ring. Hence it is reduced to showing that $\hat{\fa}^m\subset (\bx)$ for some positive integer $m$. It is easy to see both $I$ and $\hat{\fa}$ are contained in Jacobson radical of $\hat{A}$. In the mean time, every prime ideal containing $I$ or $\hat{\fa}$ is maximal. Hence both radicals of $I$ and $\hat{\fa}$ are $\prod_{i=1}^s \fm_i\hat{A}$. Thus $\hat{\fa}^m\subset (\bx)$ for certain $m>0$. This ends the proof of this lemma.   
\end{proof}

Summarizing all the above, we give the refined version of \Cref{Thm: exist_uniq_power_series}.

\begin{Thm}\label{Thm: exist_uniq_power_series_refined}
Let $\bx=\{x_1,...,x_n\}$ be a sequence of elements in $\hat{\fa}$. Write $I=(\bx)$. Let $E$ be a set of representatives of $\hat{A}/I$. Then each $y\in \hat{A}$ admits a power series expansion with respect to $\bx$ and $E$. Furthermore the following three statements are equivalent: 

(1) $E$ is finite and for any $y\in \hat{A}$, there exists a unique power series expansion of $y$ with respect to $\bx$ and $E$; 

(2) $\bx$ is a $\hat{A}$-quasi-regular and $n=d$; 

(3) $\bx$ is $\hat{A}$-regular and $n=d$.
\end{Thm}

\begin{proof}
The existence of power series expansion with respect to $\bx$ and $E$ follows from \Cref{Lem: exist_power_series}. Applying \Cref{Prop: uniq_expan_eqv_quasi-reg} to $\hat{A}$ and $\hat{\fa}$, we derive that for any $y\in \hat{A}$, there exists a unique power series expansion of $y$ with respect to $\bx$ and $E$ if and only if $\bx$ is $\hat{A}$-quasi-regular, if and only if $\bx$ is $\hat{A}$-regular. Then by \Cref{Lem: max_reg_seq_eqv_fin_quot}, (1), (2) and (3) are equivalent. 
\end{proof}

\begin{Rmk}
 \Cref{Thm: exist_uniq_power_series} follows immediately from \Cref{Thm: exist_uniq_power_series_refined} and \Cref{Lem: dim_height_depth_in_completion}. {Actually we see that to ensure the finiteness of $E$ and uniqueness of the expansions in \Cref{Thm: exist_uniq_power_series}, our assumption $\depth \fa=d$ is necessary.}
\end{Rmk}

\section{automaticity and algebraicity}\label{Sect: automaticity}
In this section we first recall basic definitions and properties $q$-automatic sequences. After that, we explain how to interpret algebraicity of elements in $\hat{A}$ with respect to a fixed regular sequence and a set of representatives. By establishing several auxiliary facts we reduce the proof of \Cref{Thm: algebraic_eqv_automatic} to a key result (\Cref{Thm: free_mod_alg_version}), whose proof is postponed to \Cref{Sect:  Henselization}.
\subsection{Automatic sequences}

Let $\Delta$ be a finite set. Roughly speaking, a ($d$-dimensional) automatic sequence over $\Delta$ is a sequence generated by a deterministic finite automaton with output (DFAO). More concretely, for a positive integer $k$, let $\Sigma_k:=\{0,\dots,k-1\}$. A {\it $d$-dimensional $k$-DFAO} is a $6$-tuple $M=(Q, \Sigma_k^d, \delta, q_0, \Delta, \tau)$ where 
\begin{itemize}
    \item[] $Q$ is a finite set of states, 
    \item[] $\delta: Q \times \Sigma_k^d\rightarrow Q$ is the transition function, 
    \item[] $q_0\in Q$ is the initial state, 
    \item[] $\Delta$ is the output alphabet, 
    \item[] $\tau: Q \rightarrow \Delta$ is the output function.
\end{itemize}
For a state $q \in Q$ and a word $w=w_1\cdots w_n$ on $\Sigma_k^d$, we recursively define 
$$\delta(q,w)=\delta(\delta(q,w_1\cdots w_{n-1}),w_n).$$ 
For any $u=(u_1,\dots,u_d) \in \bN^d$, 
let $h:=\max\{\lfloor \log_k u_1 \rfloor, \dots, \lfloor \log_k u_d \rfloor\}$. That is, $u_i< k^{h+1}$ for all $1\le i\le d$, and $k^h\le u_i$ for some $i$. We view $u$ as a word on $\Sigma_k^d$ in the following way. Write each $u_i$ as 
$$u_i=\sum_{j=0}^{h} w_j(u_i)k^j,$$ 
where $w_j(u_i)\in \Sigma_k$. By our choice of $h$, such expressions are unique. Letting 
$$w_j(u):=(w_j(u_1),\dots, w_j(u_d))\in \Sigma_k^d,$$  
then $u$ is uniquely corresponding to the word $w(u):=w_h(u)\cdots w_0(u)$. With this correspondence, a $d$-dimensional $k$-DFAO $(Q, \Sigma_k^d, \delta, q_0, \Delta, \tau)$ generates an $d$-dimensional sequence on $\Delta$:  
$$\{\tau(\delta(q_0, w(u)))\}_{u\in\bN^d}.$$ 

\begin{Defn}\label{Defn: automatic_sequence}
A $d$-dimensional sequence $\{a_u\}_{u\in \bN^d}$ over a finite set $\Delta$ is called {\it $k$-automatic} if it is generated by a $d$-dimensional $k$-DFAO. A $d$-dimensional $k$-automatic sequence can also be characterized by its $k$-kernels.
\end{Defn}

\begin{Defn}
Given a sequence $\cS=\{a_u\}_{u\in \bN^d}$, integers $l\ge 0$ and $0\le v_i<k^l$ ($1\le i\le d$),  for any $u=(u_1,\dots,u_d)\in \bN^d$, let 
\begin{equation}\label{k-kernel}
b_u:=b_{(u_1,\dots,u_d)}=a_{(k^l u_1+v_1,\dots,k^l u_d+v_d)}.
\end{equation}
Then $\{b_u\}_{u\in \bN^d}$ is a subsequence of $\cS$. The {\it $k$-kernel}, denoted by $Ker_k(\cS)$, of $\cS$ is defined to be the collection of subsequences of $\cS$ given by (\ref{k-kernel}) for all $l\ge 0$ and $0\le v_i<k^l$.  
\end{Defn}

We have the equivalent description on $k$-automatic sequences. 

\begin{Lem}\label{Lem: k-kernel}
A sequence $\{a_u\}_{u\in\bN^d}$ is $k$-automatic if its $k$-kernel is a finite set.
\end{Lem} 

If $\cS=\{a_u\}_{u\in\bN^d}$ is $k$-automatic, the cardinality of its $k$-kernel is called the {\it complexity} of $\{a_u\}_{u\in\bN^d}$. For $1\le i\le r$, let $\cS_i=\{a_u^{(i)}\}_{u\in \bN^d}$ be sequences over a finite set $\Delta$. Let $a_u=(a_u^{(1)}, \dots, a_u^{(r)})$ be $r$-tuples. Write $\cS=\{a_u\}_{u\in \bN^d}$. 

\begin{Lem}\label{Lem: auto_tuples}
The sequence $\cS$ of $r$-tuples is $k$-automatic if and only if each $\cS_i$ is $k$-automatic. Furthermore the complexity of $\cS$ is no more than the product of complexities of all $\cS_i$. 
\end{Lem}

\begin{proof}
The first statement is obvious, the second statement follows immediately from 
$$Ker_k(\cS)\subset Ker_k(\cS_1)\times \cdots \times Ker_k(\cS_r).$$
\end{proof}

We also need the following result for $q$-automatic sequences over $\bF_q$. 

\begin{Lem}\cites{Salon87, Salon89}\label{Lem: Salon's thm}
Let $\{a_u\}_{u\in \bN^d}$ be a sequence over $\bF_q$. Then $\{a_u\}_{u\in \bN^d}$ is $q$-automatic if and only if its generating series $\sum_{u\in\bN^d}a_u X^u$ is algebraic over $\bF_q(X_1,\dots,X_d)$, where $\bF_q(X_1,\dots,X_d)$ stands for the field of rational functions over $\bF_q$ with variables $X_1,\dots,X_d$.    
\end{Lem}

\begin{Rmk}
Traditionally there exist two ways to prove Christol's theorem. One is to use Cartier's operators 
$$\Lambda_v \left( \sum_{u\in\bN^d} a_uX^u \right)=\sum_{u\in \bN^d} a_{qu+v}^{1/q} X^u$$ 
where $v=(v_1,\dots,v_d)$ with $0\le v_i<q$, and the monoid generated by them. The other relies on a result of Fursternburg on diagonals of rational power series. Both ways demand that the coefficients are in a field. If we consider power series in Theorem \ref{Thm: exist_uniq_power_series}, then both ways seemingly do not work. 
\end{Rmk}

\subsection{Algebraic elements} 
We follow the notations and assumptions established in \Cref{Sect: power_ser_exp}. Precisely, we denote $A$ to be a finitely generated integral domain over $\bF_q$ with $\dim A=d$. Write $\fa$ for an ideal of $A$ such that $\depth \fa=d$. Let $\hat{A}$ be the $\fa$-adic completion of $A$, $\hat{\fa}=\fa\hat{A}$ be the completion of $\fa$ in $\hat{A}$. Fix $\bx=\{x_1,\dots,x_d\}$ to be an $\hat{A}$-regular sequence of length $d$ in $\hat{\fa}$, and take $E\subset \hat{A}$ to be a set of representatives of $\hat{A}/(\bx)$. For an element $y\in \hat{A}$, 
let $y=\sum_{u\in \bN^d} a_ux^u$ be the unique power series of $y$ with respect to $\bx$ and $E$ (\Cref{Thm: exist_uniq_power_series_refined}). 
\begin{Defn}
    With the above notations,  an element $y\in \hat{A}$ is called {\it $q$-automatic with respect to $\bx$ and $E$} if the sequence of coefficients $\{a_u\}_{u\in\bN^d}$ is a $q$-automatic sequence over $E$. When $\bx$ and $E$ are clear from the contents, we simple call this sequence {\it $q$-automatic}. 
\end{Defn}

\begin{Defn}
    An element $y\in \hat{A}$ is called  {\it algebraic (over $A$)} if there exists a polynomial $P(X)\in A[X]$ such that $P(y)=0$.
\end{Defn}
Denote by $A^a$ the subset of $\hat{A}$ consisting of elements algebraic over $A$. As said in \Cref{Sect: intro}, our main goal is to generalize Christol's theorem to the expansion of $y$ with respect to $\bx$ and $E$. That is, to show an element of $\hat{A}$ is $q$-automatic if and only if it is algebraic over $A$. Such an equivalence apparently depends on the choice of $\bx$ and $E$. It would be unreasonable to consider this equivalence if there exist non-algebraic elements in $\bx$ or $E$. So we always assume the elements of $\bx$ and $E$ are algebraic over $A$. 

The most interesting case happens when we choose $\bx$ to be a sequence of algebraic elements, which is possible due to 
\Cref{Lem: dim_height_depth_in_completion}. We also pay special attention to the situations that $E$ consists of algebraic elements. Such an $E$ exits. In fact, suppose $\hat{A}/(\bx)$, as an $\bF_q$-vector space, is of dimension $r$. Let $\bar{e}_1,\cdots,\bar{e}_r$ be an $\bF_q$-basis of $\hat{A}/(\bx)$. For each $\bar{e}_i$, choose $e_i\in A^a$ such that the reduction of $e_i$ modulo $(\bx)$ is $\bar{e}_i$. Such an $e_i$ exists because $A^a$ is dense in $\hat{A}$. Then by letting $E=\bF_qe_1+\cdots+\bF_qe_r$, we get a set $E$ with algebraic elements. Actually the set $E$ chosen in this way forms an $\bF_q$-vector space. 

Define $B=\bF_q[x_1,\dots, x_d]$ to be the subring of $\hat{A}$ that is generated by adjoining elements $x_1,\dots,x_d$ to $\bF_q$. Let $\fb=(x_1,\cdots,x_d)$ be the maximal ideal of $B$ generated by $x_1,\cdots,x_d$. One can check that $\fb=\hat{\fa}\cap B=\hat{\fm} \cap B$ for any maximal ideal $\hat{\fm}$ of $\hat{A}$. 

\begin{Lem}\label{Lem: B_poly_ring}
    The Krull dimension of $B$ is $d$. In particular, $x_1, \dots, x_d$ are algebraically independent over $\bF_q$, and $B$ is isomorphic to the polynomial ring of $d$ indeterminates over $\bF_q$.  
\end{Lem}
\begin{proof}
    Only the first statement requires a proof. Since it is easy to see that the $\dim B$ is less than or equal to $d$, we are reduced to showing $\dim B\geq d$. For this, note that since $\{x_1, \dots, x_d\}$ is a regular sequence of $A$ in $\hat{\fa}$. Take $\hat{\fm}$ to be any maximal ideal containing $\hat{\fa}$, then  $\{x_1, \dots, x_d\}$ is again a regular sequence in the localization $\hat{A}_{\hat{\fm}}$. In fact, according to \Cref{Defn: reg_seq}, we know that $\{x_1, \dots, x_d\}$ is regular if and only if the morphism of $\hat{A}$-module 
    \[
    x_{i+1}: \hat{A}/(x_1,\dots, x_i) \to \hat{A}/(x_1,\dots, x_i) , \quad \bar{a}\mapsto \bar{a}x_{i+1}
    \]
    is injective  for every $i=0, \dots, d-1$ (when $i=0$, we take $x_1: \hat{A}\to \hat{A}$). Recall that $\hat{A}_{\hat{\fm}}$ is a flat $\hat{A}$-module, hence tensoring with $\hat{A}_{\hat{\fm}}$ will preserves the injectivity of the induced map 
    \[
     x_{i+1}: \hat{A}_{\hat{\fm}}/(x_1,\dots, x_i) \to \hat{A}_{\hat{\fm}}/(x_1,\dots, x_i).
    \]
    This tells that $\{x_1, \dots, x_d\}$ is also regular in the local ring $\hat{A}_{\hat{\fm}}$. In particular, it is Cohen-Macaulay for $\dim \hat{A}_{\hat{\fm}}=\dim A=d$ (cf. \Cref{Defn: Cohen-Macaulay}).

    Now consider the natural homomorphism of local rings $B_\fb\to \hat{A}_{\hat{\fm}}$. According to the inequality \cite{Eisenbud}*{Theorem~10.10}
    \[
    d=\dim \hat{A}_{\hat{\fm}}\leq \dim B_\fb + \dim \hat{A}_{\hat{\fm}}/\fb \hat{A}_{\hat{\fm}}
    \]
    it suffices to show $ \dim \hat{A}_{\hat{\fm}}/\fb \hat{A}_{\hat{\fm}}=0$. Notice that  
    $
    \hat{A}_{\hat{\fm}}/\fb \hat{A}_{\hat{\fm}}=\hat{A}_{\hat{\fm}}/(\bx)
    $. The regularity of the sequence $\{x_1, \dots, x_d\}$ in $\hat{A}_{\hat{\fm}}$ implies that $\dim \hat{A}_{\hat{\fm}}/(\bx)=0$ by \Cref{Prop: CM_basic_prop}(3). Hence the desired equality $\dim B=\dim B_\fb=d$ follows. 
\end{proof}

Let $\hat{B}$ be the completion of $B$ in $\hat{A}$, which coincides with $\fb$-adic completion of $B$. The above lemma tells that  $\hat{B}$ is the power series ring $\bF_q[[x_1,\dots,x_d]]$. Write $\hat{\fb}:=\fb \hat{B}$. 

\begin{Lem}\label{Lem: free_mod}
Assume $E$ is an $\bF_q$-vector space. Let $e_1,\dots,e_r$ be an $\bF_q$-basis of $E$. Then $$\hat{A}=\hat{B}e_1\oplus \cdots \oplus \hat{B}e_r$$
as $\hat{B}$-modules. In particular, $\hat{A}$ is a free $\hat{B}$-module of rank $r=\dim_{\bF_q} \hat{A}/(\bx)$.   
\end{Lem} 

\begin{proof}
For any $y\in \hat{A}$, by Theorem \ref{Thm: exist_uniq_power_series}, there exists a unique expansion 
$$ y=\sum_{u \in \bN^d}a_ux^u$$
with $a_u\in E$. Write $a_u$ as $a_u=a_u^{(1)}e_1+\cdots +a_u^{(r)}{e_r}$ with $a_u^{(i)}\in \bF_q$. Then 
$$y=\sum_{u\in \bN^d} \sum_{i=1}^r a_u^{(i)}e_i x^u=\sum_{i=1}^r \sum_{u\in \bN^d}a_u^{(i)}x^u.$$
Set $y^{(i)}:=\sum_{u\in \bN^d}a_u^{(i)}x^u\in \bF_q[[x_1,\dots,x_d]]=\hat{B}$. Then for every $y\in \hat{A}$, there exist unique $y^{(1)},\dots,y^{(r)}\in\bF_q[[x_1,\dots,x_d]]$ such that $y=y^{(1)}e_1+\cdots y^{(r)}e_r$. Hence we have  $\hat{A}=\hat{B}e_1\oplus \cdots \oplus \hat{B}e_r$.
\end{proof}

\begin{Lem} \label{Lem: auto_seq_eqv_alg_summand}
Let $y\in \hat{A}$ and $y=y^{(1)}e_1+\cdots y^{(r)}e_r$ with $y^{(i)}\in \hat{B}$. Then 
$y$ is $q$-automatic if and only if each $y^{(i)}$ ($1\le i\le r$) is algebraic.
\end{Lem}

\begin{proof}
Let $\{a_u\}_{u\in \bN^d}$ and $\{a_u^{(i)}\}$ be as in the proof of \Cref{Lem: free_mod}. Then by \Cref{Lem: auto_tuples}, $\{a_u\}_{u\in \bN^d}$ is $q$-automatic if and only if for all $1\le i\le r$, $\{a_u^{(i)}\}$ are $q$-automatic, if and only if $y^{(i)}$ are algebraic. 
\end{proof}

By Lemma \ref{Lem: auto_seq_eqv_alg_summand}, to prove Theorem \ref{Thm: algebraic_eqv_automatic}, it is equivalent to show $y$ is algebraic over $A$ if and only if all $y^{(i)}$ are algebraic over $B$. Let $B^a$ denote the subring of $\hat{B}$ consisting of elements algebraic over $B$. The following result is key to the proof of Theorem \ref{Thm: algebraic_eqv_automatic}. 
\begin{Thm}\label{Thm: free_mod_alg_version}
$A^a=B^ae_1\oplus \cdots \oplus B^ae_r$ is a free $B^a$-module of rank $r$. 
\end{Thm}

The proof of Theorem \ref{Thm: free_mod_alg_version} is left to the next section. To study $A^a, B^a$ and their relationship, we will introduce the henselization of a ring at an ideal. In conclusion, we see that:   

\subsection{\texorpdfstring{Proof of \Cref{Thm: algebraic_eqv_automatic}}{Proof of Theorem~1.2}}\label{Sect: pf_thm_1.2}
Suppose $y$  to be an element of $\hat{A}$. Expand $y=y^{(1)}e_1+\cdots y^{(r)}e_r$, $y^{(i)}\in \hat{B}$ according to \Cref{Lem: free_mod}. Then by \Cref{Lem: auto_seq_eqv_alg_summand} we know that $y$ is $q$-automatic if and only if each $y^{(i)}$ is algebraic over $B$. Then using \Cref{Thm: free_mod_alg_version} one finds that this happens precisely when $y$ itself is algebraic over $A$. This establishes the desired equivalence and finishes the proof of the theorem.

\section{\texorpdfstring{Henselization and proof of \Cref{Thm: free_mod_alg_version}}{Henselization}}\label{Sect:  Henselization}
The main purpose of this section is to prove \Cref{Thm: free_mod_alg_version}. 
To that end, we begin by recalling the notion of the henselization of a pair. 
After establishing the necessary properties of this construction, we prove in 
\Cref{Prop: A^h=A^a} that the algebraic closure \(A^{a}\) of \(A\) in \(\hat{A}\) 
coincides precisely with the henselization of \(A\). This identification provides 
a crucial bridge, allowing us to translate the statement of 
\Cref{Thm: free_mod_alg_version} into an equivalent formulation for henselian rings. 
In view of the close analogy between henselization and completion, the proof of 
\Cref{Thm: free_mod_alg_version} is then completed in \Cref{Sect: proof_of_free_B^a_mod}.
\subsection{Henselization and algebraic closure}
In this section, a {\it pair} $(A,\fa)$ consists of a ring $A$ and an ideal $\fa$ in $A$. A \emph{morphism of pairs} $\varphi: (A,\fa)\to (A',\fa')$ is a ring homomorphism from $A$ to $A'$ such that $\varphi(\fa)\subset \fa'$.

\begin{Prop}\label{Prop: Henselian_pair}
    Let $(A, \fa)$ be a pair, for a polynomial $f\in A[T]$ let $\bar{f}$ denote the corresponding residue polynomial in the quotient ring $(A/\fa)[T]$. Then the following are equivalent: 
    \begin{enumerate}
        \item The ideal $\fa$ is contained in the Jacobson radical of $A$. For any monic polynomial $f\in A[T]$ and a factorization $\bar{f}=g_0h_0$, with $g_0$ and $h_0$ coprime monic polynomials in the quotient ring $(A/\fa)[T]$, there exists a factorization $f=gh$ in $A[T]$ with $g,h$ monic and $\bar{g}=g_0$, $\bar{h}=h_0$. 
        \item Given an \'etale ring homomorphism $A\to A'$ and an $A$-algebra map $\sigma: A'\to A/\fa$, there exists an $A$-algebra map $A'\to A$ lifting $\sigma$. 
        \item For any finite $A$-algebra $B$ the map $B\to B/IB$ induces a bijection on idempotents. 
    \end{enumerate}
\end{Prop}
\begin{proof}
    See \cite{stacks-project}*{\href{https://stacks.math.columbia.edu/tag/09XI}{Lemma 09XI}} or \cite{Raynaud}*{XI, \S2, Proposition~1}. 
\end{proof}

\begin{Defn}
    A pair $(A,\fa)$ satisfying any one of the equivalent conditions in \Cref{Prop: Henselian_pair} is called a \emph{henselian pair}. 
\end{Defn} 

If $A$ is $\fa$-adically complete, then $(A,\fa)$ is a henselian pair (\cite{stacks-project}*{\href{https://stacks.math.columbia.edu/tag/0ALJ}{Lemma 0ALJ}}). For a general pair $(A, \fa)$, which is not necessarily henselian, one can talk about the ``smallest henselian pair'' that contains $(A, \fa)$. Strictly speaking, we consider the next concept. 

\begin{Defn}\label{Defn: henselization_pair}
    Let $(A, \fa)$ be a pair, a \emph{henselization} of $(A, \fa)$ is a henselian pair $(A^h, \fa^h)$ together with a morphism of pairs $\iota: (A, \fa)\to (A^h, \fa^h)$ which satisfies the universal property: 
    
    For every henselian pair $(B, \fb)$ and every morphism $v: (A, \fa)\to (B, \fb)$, $v$ factors uniquely through a morphism $u: (A^h, \fa^h)\to (B, \fb)$, i.e., there is a unique morphism of pair $u$ such that the following diagram is commutative 
    \[
    \begin{tikzcd}
        (A, \fa) \arrow[r, "\iota"] \arrow[rd, "v"']& (A^h, \fa^h) \arrow[d, "\exists ! u"]\\
        & (A, \fa)
    \end{tikzcd}
    \]
    We will usually denote the henselization of $(A, \fa)$ by $(A, \fa)^h$, or even by $A^h$ if there is no danger of confusion. It is easy to see that, when exists, the henselization of a pair is unique up to isomorphism.
\end{Defn}

For the existence of the henselization, we recall the concept of colimits of diagrams and \'etale neighborhoods of pairs. 

\begin{Defn}
Let $\cI$ and $\cC$ be categories. A ($\cI$-)\emph{diagram} in $\cC$ is a functor $M: \cI \to \cC$. We call \emph{index category} of $M$. For each object $i\in \cI$, denote by $M_i$ the object in $\cC$ corresponding to $i$. The \emph{colimit} of the $\cI$-diagram $M$ in $\cC$ is an object $\colimit_{\cI} M$ (or simply $\colimit M$) of $\cC$ together with morphisms $s_i: M_i \to \colimit M$ indexed by objects $i$ in $\cI$ such that 
\begin{enumerate}
    \item for $\phi: i \to i'$ a morphism in $I$, we have $s_i =s_{i'} \circ M(\phi)$; 
    \item for any object $W$ in $\cC$ and any family of morphisms $t_i: M_i \to W$ indexed by $i\in Ob(\cI)$ such that for all $\phi: i \to i'$ in $\cI$ we have $t_{i}=t_{i'} \circ M(\phi)$, there exists a unique morphism $t: \colimit M \to W$ such that $t_i=t \circ s_i$ for every object $i$ of $\cI$.   
\end{enumerate}
\end{Defn} 

\begin{Defn}
A pair over $(R, I)$ is a pair $(S, J)$ together with a morphism $\varphi: (R, I) \to (S, J)$. A pair $(S, J)$ over $(R, I)$ is called an \emph{\'etale neighborhood} of $(R, I)$ if $S$ is an \'etale $R$-algebra, $J=IS$, and the induced residue ring map $R/I \to S/J$ is isomorphic.    
\end{Defn}

Given a pair $(R, I)$, let $\cC$ be the category consisting of pairs over $(R, I)$, and $\cD$ the category consisting of \'etale neighborhoods of $(R, I)$. Then $\cD$ is a subcategory of $\cC$. Denote by $F: \cD \to \cC$ the inclusion functor. Then 

\begin{Lem}\label{Lem: henselization_local}
    Fixing a pair $(R, I)$, let $\cC$ and $\cD$ be as above. 
    \begin{enumerate}
        \item The category $\cD$ is directed. 
        \item $R^h=\colimit F$.
    \end{enumerate}
\end{Lem} 

\begin{proof}
See \cite{stacks-project}*{\href{https://stacks.math.columbia.edu/tag/0A02}{Lemma 0A02}} and its proof. 
\end{proof}

In the following, we fix a pair $(R,I)$ and one of its \'etale neighborhoods $(R', I')$. Let $\cD'$ be the category of \'etale neighborhoods of $(R', I')$. Since any \'etale neighborhood of $(R', I')$ is also an \'etale neighborhood of $(R, I)$. The category $\cD'$ is a subcategory of $\cD$. Let $G: \cD' \to \cD$ be the corresponding inclusion functor. The following fact is well known to the experts (at least for local case). For the convenience of the reader, a proof is given. 

\begin{Lem}\label{Lem: henselization_of_etale_neighborhood}
Let $(R, I)$, $(R', I')$ be as the above, and let $\cD$ and $\cD'$ be the associated categories. The inclusion functor $G$ is a final functor (See \cite{MacLane} for definition of final functors). Hence 
$$(R')^h:=\colimit (F \circ G) = \colimit F=:R^h.$$
\end{Lem}

\begin{proof}
The isomorphism in the middle of the equation is an immediate consequence if $G$ is a final functor. Hence the proof is reduced to checking that $G$ is final. For this, we will need to show the following: 
\begin{enumerate}
    \item for each object $(S, J)$ of $\cD$, there exists an object $(S',J')$ of $\cD'$ and a morphism $(S, J)\to G(S', J')=(S', J')$ in $\cD$.
    \item for any two such morphisms, say $\phi': (S, J)\to (S', j')$ and $\phi'': (S, J)\to (S'', j'')$, there exist finitely many objects $\{(S'_i, J'_i)\}_{i=0, \dots, n}$ such that 
    \begin{enumerate}
        \item $n$ is even, $(S_0', j'_0)=(S', J')$, $(S'',J'')=(S'_n, J'_n)$;
        \item for each even integer $m\in \{1, \dots, n\}$, there are two morphisms $f_{m-1}, (S_m', J_m')\to (S_{m-1}', J_{m-1}')$ and $f_{m}: (S_m', J_m')\to (S_{m+1}', J_{m+1}')$
        \item for each $i$, a morphism $\phi_i: (S, J)\to G(S_i', J_i')=(S_i', J_i')$; moreover, $\phi_1=\phi'$, $\phi_n=\phi''$
    \end{enumerate}
    and the above objects together with the morphisms fit into a commutative diagram as follow 
    \[
    \begin{tikzcd}
        &&&(S, J)\arrow[dlll, "\phi_0=\phi'"']\arrow[dll, "\phi_1"']\arrow[dl, "\phi_2"']\arrow[rd, "\phi_{2k}"]\arrow[drrr, "\phi_n"]&&&\\
        (S'_0, J'_0) \arrow[r, "f_0"']& (S_1', J_1') & (S_2', J_2')\arrow[l, "f_1"] \arrow[r, "f_2"']& \cdots & (S_{2k}', J_{2k}') \arrow[l, "f_{2k-1}"] \arrow[r, "f_{2k}"']& \cdots & (S_n', J_n')\arrow[l, "f_{n-1}"]
    \end{tikzcd}
    \]
\end{enumerate}
For (1), take any object $(S,J)$ of $\cD$, let $S'=S \otimes_R R'$ and $J'=I'(S \otimes_R R')$. Then by basic property of \'etale morphism is stable under tensor product, we know that $S'$ is \'etale over $R'$. Also notice that $S'/J'=(S\otimes_R R')/(S\otimes_R I')=S\otimes_R R'/I'=S/J$. Hence $(S', J')$ is an object of $\cD'$. Moreover, the tensor construction of $(S', J')$ induces a morphism $(S, J) \to (S' J')=G(S', J')$. 

As for (2), suppose  $\phi': (S, I) \to (S', J')$ and $\phi'': (S, I) \to (S'', J'')$ to be two morphisms in $\cD$, where $(S', J')$ and $(S'', J'')$ are objects of $\cD'$. Consider the tensor product $S'\otimes_{S \otimes_R R'} S''$. We claim that $S'\otimes_{S \otimes_R R'} S''$ is \'etale over $R'$. Indeed, in the commutative diagram 
\[
\begin{tikzcd}
    & & S'\arrow[rd, "g'"] & \\
    R' \arrow[r, "\id\otimes \iota"]  & R'\otimes_R S \arrow[rd, "g\otimes \phi''"'] \arrow[ru, "f \otimes \phi'"]  & & S'\otimes_{S \otimes_R R'} S''\\
    & & S'' \arrow[ru, "f'"']&
\end{tikzcd}
\]
every homomorphism is \'etale. Here $\iota: R\to S$, $f: R'\to S'$ and $g: R'\to S''$ are the natural homomorphisms give rise the \'etale structures. And  $f\otimes \phi': R'\otimes_R S\to S'$ and $g\otimes \phi'': R'\otimes_R S\to S''$ are naturally induced homomorphisms. Since both $R'\otimes_R S$ and $S'$ are all \'etale over $R'$, and since $f\otimes \phi'$ is a homomorphism $R'$-algebras, thus it is \'etale by the standard property of \'etale maps (see \cite{stacks-project}*{\href{https://stacks.math.columbia.edu/tag/00U7}{Lemma 00U7}}). By the same argument $g\otimes \phi''$ is also \'etale. Then as the homomorphisms induced by tensor product, $f'$ and $g'$ are also \'etale. This verifies our claim. 

Now consider the pair $(S_1', J_1'):=(S'\otimes_{S \otimes_R R'} S'', I'(S'\otimes_{S \otimes_R R'} S''))$, note that 
\[
S'\otimes_{S \otimes_R R'} S''/I'(S'\otimes_{S \otimes_R R'} S'')=S'/I'S'\otimes_{S\otimes_R R'} S''= (S\otimes_R R')/I(S\otimes R')\otimes_{S\otimes_R R'} S''=S''/I''=R'/I'. 
\]
This tells that $(S_1', J_1')$ is an \'etale neighborhood of $(R', I')$, hence an object in $\cD'$.

We take $n=2$, and denote $(S_i', J_i')$ $(i=0,1,2)$ as (a). Consider the next diagram 
\[
\begin{tikzcd}
(S, J) \arrow[r, equal] \arrow[d, "\phi'"] & (S, J) \arrow[d, "\phi_1"] & (S, J) \arrow[l, equal] \arrow[d, "\phi''"] \\
(S', J') \arrow[r, "g'"]  & (S'\otimes_{S \otimes_R R'} S'', I'(S' \otimes_{S \otimes_R R'} S'')) 
& (S'', J'') \arrow[l, "f'"'] 
\end{tikzcd} 
\]
where the homomorphism $\phi_1$ is the composite of $S\to R'\otimes_R S\overset{f}{\to}S' \overset{g'}{\to}S'\otimes_{S \otimes_R R'} S''$. Thus the  commutativity of this diagram follows from that of the previous diagram. Hence this commutative diagram matches the requirements of (b) and (c) of (2). As an immediate consequence, $G: \cD' \to \cD$ is a final functor, and this ends the proof of the lemma. 
\end{proof}

\begin{Lem}\label{Lem: completion_henselization_flat_I}
Let $A$ be a Noetherian ring, let $\fa$ be an ideal of $A$. Let $(A^h, \fa^h)$ and $(\hat{A}, \hat{\fa})$ be the associated henselization and completion respectively. 
Then: 
  \begin{enumerate} 
    \item $\fa^h=\fa A^h$, $A/\fa^i\cong A^h/\fa^iA^h$ for any $i>0$; 
    \item $\widehat{A^h}\cong \hat{A}$, where $\widehat{A^h}$ is the $\fa^h$-adic completion of $A^h$; 
    \item $A^h$ is flat over $A$, $\hat{A}$ is faithfully flat over $A^h$; 
    \item if $\fa$ is contained in the Jacobson radical of $A$, then both $A^h$ and $\hat{A}$ are faithfully flat over $A$.
  \end{enumerate}
\end{Lem}
\begin{proof}
    See \cite{stacks-project}*{\href{https://stacks.math.columbia.edu/tag/0AGU}{Lemma 0AGU}} for (1), \cite{stacks-project}*{\href{https://stacks.math.columbia.edu/tag/0AGV}{Lemma 0AGV}} for 
    (2), (3), and \cite{Greco}*{Theorem~6.5} for (4). 
\end{proof}

Since the map $A^h\rightarrow \hat{A}$ is faithfully flat, it is injective. If $\bigcap_{i=1}^{\infty}\fa^i=(0)$, then the map $A\rightarrow \hat{A}$ is injective, which implies the injectivity of the morphism $A\rightarrow A^h$. In this case we consider $A$ and $A^h$ as subrings of $\hat{A}$. In particular the condition $\bigcap_{i=1}^{\infty}\fa^i=(0)$ holds when $A$ is a Noetherian domain or a Noetherian local ring. 

Recall that in \Cref{Sect: automaticity}, if $A$ is a domain, we let $A^a$ be the {\it algebraic closure} of $A$ in $\hat{A}$, i.e., the set of elements of $\hat{A}$ that are annihilated by a nonzero polynomial over $A$. The next result reveals the relation between $A^h$ and $A^a$. 

\begin{Prop}\label{Prop: A^h=A^a}
Let $A$ be a finitely generated integral domain over a field, $\fa$ an ideal of $A$ satisfying $\depth \fa=\dim A=d$. Let $(A^h, \fa^h)$ and $(\hat{A},\hat{\fa})$ be henselization and completion of $(A,\fa)$ respectively, and let $A^a$ be the algebraic closure of $A$ in $\hat{A}$. Then $A^h=A^a$.   
\end{Prop}

\Cref{Prop: A^h=A^a} should be well known to experts. However, due to the lack of reference, we will provide a proof to this result at the end of this subsection after some preparation.

Recall that a Noetherian ring $R$ is said to be \emph{excellent} if it is a universally catenary G-ring, and for every finitely generated $R$-algebra $R'$, the regular locus ${\rm Reg}(R')$ in $\Spec(R')$ is open. To save the length of this paper, we refer the interested reader to the classical book of Matsumura \cite{Matsumura-commutative-ring-theory-1986}*{\S34} for the detailed definitions of the terminologies. The needed properties about excellent rings are cited below. 

\begin{Prop}\label{Prop: excellent}
    \begin{enumerate}
        \item The class of excellent rings contains fields, Noetherian complete local rings and the integer ring $\bZ$.
        \item The class of excellent rings is closed under localization, finitely generated extensions of rings, and passing to quotients. 
        \item Assume $R$ to be an excellent ring, and let $S$ be a finitely generated integral $R$-algebra. Take $S'$ to be the integral closure of $R$ in the fraction field of $S$. Then $S'$ is a finitely generated module of $R$. 
    \end{enumerate}
\end{Prop}

The next key lemma is due to Artin, we cite it in the form that is suitable for our application. 

\begin{Lem}\cite{Artin-approximation}*{Theorem~(I.10), Artin’s approximation theorem}
    Let $R$ be a field or an excellent discrete valuation ring, and let $(A, \fp)$ be a pair such that $A$ is an $R$-algebra of finite type, and $\fp$ is a prime ideal of $A$. Take $(A^h, \fp^h)$ and $(\hat{A}, \hat{\fp})$ to be the henselization and completion respectively. Given an arbitrary system of polynomial equations 
    \begin{equation}
        f(T)=0, \quad T=(T_1, \dots, , T_N)
    \end{equation}
    with coefficients in $A^h$, a solution $\hat{t}=(\hat{t}_1, \dots, \hat{t}_N)$ with $\hat{t}_i\in \hat{A}$, and an integer $c$, there exists a solution $t=(t_1, \dots, t_N)$ with $t_i\in A^h$ with 
    \[
    t_i\equiv \hat{t}_i \ (\bmod\ \fp^c) \text{ for all }i. 
    \]
\end{Lem}

As the first step of proving \Cref{Prop: A^h=A^a}, we consider the local case. Until \Cref{Prop: algebraic_closure_revisited}, we make the following assumption: 
\begin{center}
    \emph{$A$ is a local domain with maximal ideal $\fm$ which comes from a localization of a finitely generated algebra over a field. }
\end{center}

\begin{Prop}\label{Prop: algebraic_closure}
Take $(A^h, \fm^h)$ and $(\hat{A}, \hat{\fm})$ to be the henselization and completion of $(A, \fm)$ respectively. Assume that $A$ is a normal integral domain, then  $A^h$ is the algebraic closure of $A$ in $\hat{A}$. 
\end{Prop}
\begin{proof}
    First, according to the construction of $A$, we know it is excellent and normal, hence $\hat{A}$ is a domain.  Notice that if $B$ is an \'etale $A$ algebra, then $B$ is finite $A$-module. Thus every element $b\in B$ is algebraic over $A$. Hence $A^h$ is  algebraic over $A$ by \Cref{Lem: henselization_local}(2).  
    
    To show $A^h$ is the algebraic closure of $A$ in $\hat{A}$, we still need to show that $A^h$ is algebraically closed inside $\hat{A}$. Assume $\hat{y}\in \hat{A}$ to be an element which satisfies the polynomial equation
    \[
    f(\hat{y})=0, \quad f(t)\in A^h[t].
    \] 
    By Artin's approximation theorem above, for every positive integer $c$, there exists an element $y^h_c\in A^h$ satisfying  
    \begin{equation}\label{Eqn: Artin_approx}
        f(y^h_c)=0, \quad y^h_c\equiv \hat{y}\ \bmod\ (\fm^h)^c.
    \end{equation}
    By the fundamental theorem of algebra we know that $f(t)$ has at most $\deg(f)$ many roots in $\hat{A}$ or its fraction field. Hence as $c$ increases (for example, when $c\gg 0$ so that $|(\fm^h)^c|$ is smaller than the distance of any two distinct roots of $f(t)$), \eqref{Eqn: Artin_approx} tells that $\hat{y}=y^h_c$ is inside $A^h$. Hence every polynomial over $A^h$ has a root inside $A^h$, i.e., $A^h$ is algebraically closed in $\hat{A}$.  
\end{proof}

Next, we want to remove the ``normal'' condition in \Cref{Prop: algebraic_closure}. For this, we consider a local pair $(A,\fm)$ such that $A$ is not necessarily normal. Let $C$ be the integral closure of $A$ in its quotient field. Then $C$ is a normal integral domain, and is a finite $A$-module by \Cref{Prop: excellent}(3). Write $\fc:=\fm C $. Let $\hat{C}$ (resp. $C^h$) be the completion (resp. henselization) of $C$ at $\fc$. We show that the algebraic closure $C^a$ of $C$ in $\hat{C}$ coincides with $C^h$. Two facts on henselizations are needed. 

\begin{Lem}\label{Lem: henselization_radical_intersection}
Let $(R,I)$ be a pair, $(R^h, I^h)$ be the henselization of $(R, I)$. 
\begin{enumerate}
    \item If $(S,J)$ is the henselization of $(R, \sqrt{I})$, then $(R^h,I^h)=(S,IS)$, and $J=\sqrt{IS}$. 
    \item If $I=I_1 \cap \cdots \cap I_n$, where $I_1,\dots, I_n$ are ideals of $R$ such that $I_i+I_j=R$ for all $1\le i<j\le n$, then the canonical map 
    $$(R,I)^h \to \prod_{i=1}^n (R,I_i)^h$$
    is an isomorphism. 
\end{enumerate}
\end{Lem}

\begin{proof}
(1) See \cite{Greco}*{Proposition 8.6}. (2) See \cite{stacks-project}*{\href{https://stacks.math.columbia.edu/tag/0H7Q}{Lemma 0H7Q}}. 
\end{proof}

\begin{Lem}
    Under the above setups, $C^h$ is the algebraic closure of $C$ in $\hat{C}$. 
\end{Lem}
\begin{proof}
Since $\fm$ is the maximal ideal of $A$ and $C$ is a finite $A$-algebra, the prime ideals of $C$ containing $mC$ are all maximal. Suppose they are denoted as $\fn_1, \dots, \fn_n$, then the radical $\sqrt{\fm C}$ of $\fm C$ is $\prod_{i=1}^n \fn_i$. For $1\le i\le n$, let $\hat{C}_i$ be the completion of $C$ at $\fn_i$ respectively. It is easy to see that 
$$
\hat{C}\cong \prod_{i=1}^n \hat{C}_i.
$$ 
It then follows from the definition of algebraic closure that $C^a$ is isomorphic to the direct product of the algebraic closures $C^a_i$ of $C$ in $\hat{C}_i$. Since $C$ is normal, $C^a_i=C^h_i$ by \Cref{Prop: algebraic_closure}, where $C^h_i$ is the henselization of $(C, \fn_i)$. On the other hand, by \Cref{Lem: henselization_radical_intersection}(2) we have 
$$C^h \cong \prod_{i=1}^k C_i^h.$$ 
Hence the lemma follows from the equality
$$C^a=\prod_{i=1}^k C^a_i=\prod_{i=1}^k C^h_i=C^h.$$  
\end{proof}

We will need the (faithful) flatness of completion and henselization. 

\begin{Lem}\label{Lem: completion_henselization_flat}
With $(A, \fm)$, $(C, \fc)$ and $(C_i, \fn_i)$ as the above, we have 
    \begin{enumerate}
        \item $C$ is faithfully flat over $A$;
        \item $A^h$ and $\hat{A}$ are faithfully flat over $A$.
        \item $C^h$ and $\hat{C}$ are faithfully flat over $C$.
    \end{enumerate}
\end{Lem}
\begin{proof}
    By the construction of $C$, it is a finite module of $A$ by \Cref{Prop: excellent}(3), hence is flat over $A$. Moreover, due to the going-up theorem \cite{Atiyah-Macdonald}*{Theorem~5.11} the natural morphism $\Spec(C)\to \Spec(A)$ is surjective. Thus the faithfully flatness follows from the fact that a surjective flat morphism is faithfully flat \cite{MilEC}*{Proposition~2.7}. This proves (1). The rest statements are immediate consequences of \Cref{Lem: completion_henselization_flat_I}(4).
\end{proof}

In general, for an ideal $I$ of a ring $R$ (not necessarily Noetherian) and an $R$-module $M$, the natural composition 
\[
\hat{R}\otimes_R M\to \hat{R}\otimes_R \hat{M}\to \hat{R}\otimes_{\hat{R}} \hat{M}=\hat{M}
\]
needs not to be either injective or surjective. But under certain finiteness conditions we have the following.

\begin{Lem}\cite{Atiyah-Macdonald}*{Proposition~10.13}\label{Lem: completion_base_change}
    If $R$ is Noetherian and $M$ is a finitely generated module of $R$, then the above map $\hat{R}\otimes_R M\to \hat{M}$ is isomorphic. 
\end{Lem}

There is a similar result for henselization. 

\begin{Lem}\cite{stacks-project}*{\href{https://stacks.math.columbia.edu/tag/0DYE}{Lemma 0DYE}}\label{Lem: henselization_base_change}
    Let $(R, I)\to (S, J)$ be a map of pairs such that the prime ideals of $S$ containing $IS$ are in bijection with those containing $J$. Let $(R^h, I^h)\to (S^h, J^h)$ be the induced map of henselizations. If $S$ is integral over $R$, then the naturally induced map $R^h\otimes_R S\to S^h$ is an isomorphism. 
\end{Lem}

Since $A$ is a subring of $C$, and $C$ is a finitely generated module of $A$, using \Cref{Lem: completion_henselization_flat} and \Cref{Lem: completion_base_change} one sees that the naturally induced map $\hat{A}\hookrightarrow \hat{C}$ is injective. Using the same arguments, with \Cref{Lem: completion_base_change} replaced by \Cref{Lem: henselization_base_change}, one also gets injection $A^h\hookrightarrow C^h$. Notice that $A^h\hookrightarrow \hat{A}$ is also injective. Consider the following commutative diagram of additive groups:
\begin{equation}\label{Eqn: 21}
    \begin{tikzcd}
    A^h \arrow[r, hookrightarrow ] \arrow[d, hookrightarrow]& C^h \arrow[r, "d^h"] \arrow[d, hookrightarrow]& C^h \otimes_{A^h} C^h \arrow[d, "\iota"]\\
    \hat{A} \arrow[r, hookrightarrow]& \hat{C}\arrow[r, "\hat{d}"] & \hat{C}\otimes_{\hat{A}} \hat{C}
\end{tikzcd}
\end{equation}
where every map is the naturally defined ring homomorphism, except for $d^h$  (resp. $\hat{h}$), which is the homomorphism of $A^h$-modules (resp. $\hat{A}$-modules) induced by $c\mapsto c\otimes 1-1\otimes c$; and for the ring homomorphism $\iota$, which is induced by $c_1\otimes_{A^h}c_2\mapsto c_1\otimes_{\hat{A}} c_2$.

\begin{Lem}\label{Lem: iota_inj}
    The map $\iota$ is injective. 
\end{Lem}
\begin{proof}
    According to \Cref{Lem: completion_base_change} and \Cref{Prop: excellent}(3) we can write $\hat{C}=C^h\otimes_{A^h} \hat{A}$. Under this identity the map 
    \[
        \iota:  C^h \otimes_{A^h} C^h\otimes_{A^h} A^h\cong C^h \otimes_{A^h} C^h \to  (C^h\otimes_{A^h} \hat{A})\otimes_{\hat{A}} (C^h\otimes_{A^h} \hat{A})\cong (C^h\otimes_{A^h} C^h)\otimes_{A^h} \hat{A}
    \]
    is induced by sending elements $(c_1\otimes_{A^h}c_2)\otimes_{A^h} a$ to $c_1\otimes_{A^h}c_2\otimes_{A^h} a$, i.e., $\iota={\rm id}_{C^h\otimes_{A^h} C^h}\otimes {\rm id_{A^h}}$. 
    
    Since $A^h \hookrightarrow \hat{A}$ is injective, the injectivity of $\iota$ is a consequence of the faithfully flatness of $C^h\otimes_{A^h} C^h$ as an $A^h$-module. In the following commutative diagram of rings
    \[
    \begin{tikzcd}
        C^h\otimes_{A^h} C^h & C^h=C\otimes_A A^h \arrow[l] & \arrow[l]C\\
        C^h \arrow[u, "g"]& A^h \arrow[u, "f^h"] \arrow[l, "f^h"']& A \arrow[u, "f"] \arrow[l]
    \end{tikzcd}
    \]
    where $C^h=C\otimes_A A^h$ for \Cref{Lem: henselization_base_change}. The natural inclusion map $f$ is faithfully flat by \Cref{Lem: completion_henselization_flat}(1), hence so are the induced maps $f^h$ and $g$ due to the fact that faithfully flatness is stable under base change \cite{Matsumura-commutative-alg}*{(4.B)}. Hence  the composition $g\circ f^h$ is faithfully flat and our desired property is true. 
\end{proof}

\begin{Lem}\label{Lem: ff_descent}
    In \eqref{Eqn: 21} the kernel of $d^h$ (resp. $\hat{d}$) is $A^h$ (resp. $\hat{A}$). 
\end{Lem}
\begin{proof}
    By the same arguments in the proof of \Cref{Lem: iota_inj}, one deduces that $C^h$ (resp. $\hat{C}$) is faithfully flat over $A^h$ (resp. $\hat{A}$). Then the desired result follows from the classical descent theory (see \cite{MilEC}*{Proposition~2.18} for instance). 
\end{proof}

\begin{Prop}\label{Prop: algebraic_closure_revisited}
\Cref{Prop: algebraic_closure} is still true without assuming $A$ is normal.  
\end{Prop} 
\begin{proof}
    To save notations, we will identify every element of a ring in \eqref{Eqn: 21} with its image via any (composition of) maps in this diagram. To prove this result, we need to show that for every nonzero polynomial $f(T)\in A^h[T]$, each root $a$ in $\hat{A}$ is also contained in $A^h$. But this follows from a diagram chase of \eqref{Eqn: 21} and \Cref{Lem: iota_inj} and \Cref{Lem: ff_descent}. To be more precise, we consider $f(T)$ as a polynomial in $C^h[T]$. Let $a$ be a $\hat{C}$-solution of $f(T)$, then by \Cref{Prop: algebraic_closure}, $a$ is automatically in $C^h$. Since $a\in \hat{A}$ by our assumption, we have $\hat{d}(a)=0$ according to \Cref{Lem: ff_descent}. Using the injectivity of $\iota$ (by \Cref{Lem: iota_inj}) and the commutativity of the diagram, it follows that $d^h(a)=0$ as well. But by \Cref{Lem: ff_descent} again this tells $a\in A^h$. This finishes the proof. 
\end{proof}

As promised we give the proof of \Cref{Prop: A^h=A^a}, notice that now we do not assume $A$ to be local. 

{\it Proof of \Cref{Prop: A^h=A^a}.} Let $\fa=\prod_{i=1}^s \fq_i$ be the primary decomposition of $\fa$. For any $1\le i\le s$, let $\fm_i$ be the radical of $\fq_i$. Then $\fm_1, \dots, \fm_s$ are all the maximal ideals of $A$ containing $\fa$, and $\sqrt{\fa}=\fm_1 \cdots \fm_s$. Let $A_i$ be the localization of $A$ at $\fm_i$, and $A_i^h=(A_i, \fm_iA_i)^h$. Then by \Cref{Lem: henselization_radical_intersection}, 
$$A^h=(A, \sqrt{\fa})^h \cong \prod_{i=1}^s (A, \fm_i)^h \cong \prod_{i=1}^s A_i^h.$$  
Meanwhile $\prod_{i=1}^s A^a_i \cong A^a$. Since each $A_i$ is local, \Cref{Prop: algebraic_closure_revisited} implies that $A^h_i=A^a_i$. Hence we conclude that $A^h=A^a$.  \hfill$\Box$

\subsection{Proof of Theorem \ref{Thm: free_mod_alg_version}} \label{Sect: proof_of_free_B^a_mod}
We recall the settings. Let $A$ be a finitely generated integral domain over $\bF_q$, $\fa$ an ideal of $A$ such that $\depth \fa=\dim A=d$. Let $(A^h,\fa^h)$ and $(\hat{A},\hat{\fa})$ be henselization and completion of $(A,\fa)$ respectively.  Choose any $\hat{A}$-regular sequence $\bx=\{x_1,\dots, x_d\}\subset \hat{\fa}\cap A^h=\fa^h$. Let $B=\bF_q[x_1,\dots,x_d]$ be the subring generated by $\bF_q$ and $\bx$. Let $\fb$ be the maximal ideal of $B$ generated by $x_1,\dots,x_d$. Let $(B^h, \fb^h)$ and $(\hat{B},\hat{\fb})$ be henselization and completion of $(B,\fb)$ respectively. Due to \Cref{Lem: free_mod} we know that $\hat{A}$ is a finitely generated module of $\hat{B}$. It is expected that the same property holds for the henselization, i.e., $A^h$ is a finitely generated $B^h$-module, according to the analog between completion and henselization. 

Fix $(A', \fa')$ to be an \'etale neighborhood of $(A,\fa)$ such that $x_1, \dots, x_d$ are contained in $\fa'$. The existence of such a neighborhood $(A', \fa')$ is guaranteed by the fact that $(A, \fa)^h$ is a filtered colimit of \'etale neighborhoods of $(A, \fa)$. By \Cref{Lem: henselization_of_etale_neighborhood}, $(A')^h = A^h$. So it is safe to use $A^h$ to denote both henselizations of $(A, \fa)$ and $(A', \fa')$. Let $\fm_1, \dots, \fm_s$ be all the maximal ideals of $A$ containing $\fa$. Then $\fm_i'=\fm_i A', 1\le i\le s$ are the maximal ideals of $A'$ containing $\fa'$. Let $A_i'$ be the localization of $A'$ at $\fm_i'$, and $(A_i')^h$ the henselization of $(A_i', \fm_i' A_i')$. Then $(A_i')^h=(A', \fm_i')^h$. Similar to the decomposition of $A^h$ as given in the proof of \Cref{Prop: A^h=A^a}, we have 
$$A^h=(A')^h=\prod_{i=1}^s (A_i')^h.$$ 
Now to show $A^h$ is finite as a $B^h$-module, it is enough to show $(A_i')^h$ is a finite $B^h$-module for all $1\le i\le s$. For this purpose we need a result about quasi-finite ring maps. 

\begin{Defn}\cite{stacks-project}*{\href{https://stacks.math.columbia.edu/tag/00PL}{Definition 00PL}}\label{Defn: quasi_finite}
Let $R \to S$ be a  homomorphism of rings. Let $\fq$ be a prime ideal of $S$ and let  $\fp=R\cap \fq$ be a prime ideal of $R$. Denote by $\kappa(\fq)$ (resp. $\kappa(\fp)$) the fraction field of $S/\fq$ (resp. $R/\fp$). We say the ring map $R \to S$ is {\it quasi-finite} at $\fq$ if $S_{\fq}/\fp S_{\fq}$ is of dimension zero and $\kappa(\fq)/\kappa(\fp)$ is a finite field extension.  
\end{Defn}

\begin{Lem}\cite{stacks-project}*{\href{https://stacks.math.columbia.edu/tag/05WP}{Lemma 05WP}}\label{Lem: henselization_and_quasi-finite}
With the notations in \Cref{Defn: quasi_finite}, the homomorphism $R \to S$ induces a commutative diagram 
\[\begin{tikzcd}
   R \arrow[d] \arrow[r] & R_{\fp} \arrow[d] \arrow[r] &  R_{\fp}^h \arrow[d] \\ 
   S           \arrow[r] & S_{\fq}           \arrow[r] &  S_{\fq}^h 
\end{tikzcd}
\]
where $R_{\fp}^h$ and $S_{\fq}^h$ denote henselizations (at their maximal ideals) of local rings $R_{\fp}$ and $S_{\fq}$ respectively. Assume $R \to S$ is quasi-finite at $\fq$. Then the ring map $R_{\fp}^h \to S_{\fq}^h$ is finite. 
\end{Lem}

\begin{Lem} \label{Lem: $A^h$_fin_over_$B^h$}
$A^h$ is a finite $B^h$-module.      
\end{Lem}

\begin{proof}
As explained above \Cref{Defn: quasi_finite}, it is sufficient to show that, for each $i$, $(A_i')^h$ is finite over $B^h$. By \Cref{Lem: henselization_and_quasi-finite}, this is reduced to showing that the ring homomorphism $B \to A'$ is quasi-finite at $\fm_i$ for every $i$. 

By construction, we have $\dim A/\fa=0$. Since $A'/\fb A'\cong (A/\fb A)\otimes_{A} A'$. Note that $A'$ is \'etale over $A$, hence $(A/\fb A)\otimes_{A} A'$ is \'etale over $A/\fb A$ by base change. In particular $\dim A'/\fb A'=\dim A/\fb A=0$. On the other hand, $A'/\fm_i$ is clearly finite over $\bF_q=B/\fb$. Hence we conclude that the ring map $B\to A'$ is quasi-finite at $\fm_i$. It follows that $(A_i')^h$ is finite over $B^h$ and proves the lemma. 
\end{proof}

Now we are ready to give the proof of \Cref{Thm: free_mod_alg_version}. 

{\it Proof of \Cref{Thm: free_mod_alg_version}}. By \Cref{Prop: A^h=A^a}, $A^a=A^h$, hence it is equivalent to show that $A^h$ is a free $B^h$-module of rank $r=\dim_{\bF_q} \hat{A}/(\bx)$. Choose $e_1, \dots, e_r\in A^h\subset \hat{A}$ such that their reductions $\bar{e}_1, \dots, \bar{e}_r$ modulo $(\bx)$ form an $\bF_q$-basis of $\hat{A}/(\bx)$. Then $\{\bar{e}_1, \dots, \bar{e}_r\}$ is also an $\bF_q$-basis of $A^h/\fb A^h$. By \Cref{Lem: $A^h$_fin_over_$B^h$} we know that $A^h$ is a finite $B^h$-module, then Nakayama's Lemma tells that $\{e_1, \dots, e_r\}$ generates $A^h$ as a $B^h$-module. But \Cref{Lem: free_mod} tells us that $e_1, \dots, e_r$ are also a $\hat{B}$-basis of $\hat{A}$, hence are independent over $B^h$. We then derive that $e_1, \dots, e_r$ consist of a $B^h$-basis of $A^h$. Thus $A^h$ is a free $B^h$-module of rank $r$. This concludes the proof of \Cref{Thm: free_mod_alg_version}. \hfill$\Box$

\section{Applications and further discussion}\label{Sect: discussion}
In the final section of this paper, we discuss several topics related to our main theorem. We first present examples showing that, when a set of representatives \(E\) is not a vector space, the conclusion of Theorem 1.2 can either hold or fail. We then apply our theorem to the study of $\beta$-expansions. By interpreting the relevant notions within our framework, we obtain affirmative answers to several questions that have appeared in the literature.
\subsection{\texorpdfstring{Some examples}{Potential generalizations}}
Let $A$, $\fa$, $\bx$ and $E$ be as in \Cref{Thm: algebraic_eqv_automatic}. Assume $E$ to be an $\bF_q$-vector subspace of $\hat{A}$. For $f\in (\bx)$, take $E'\coloneqq E+f=\{r+f: r\in E\}$ to be another set of representatives of $\hat{A}/(\bx)$ which may not be an $\bF_q$-vector subspace. For every $y\in \hat{A}$, let 
$$
y=\sum_{u\in\bN^d} a_u'x^u,\quad  a_u'\in E'
$$ 
be the expansion of $y$ with respect to $\bx$ and $E'$. Note that $a_u'-f\in E$, then  
$$
y-f\sum_{u\in\bN^d}x^u=\sum_{u\in \bN^d} (a_u'-f)x^u
$$ 
is the expansion of $y-f\sum_{u\in \bN^d}x^u=y-f/\prod_{i=1}^d (1-x_i)$ with respect to $\bx$ and $E$. Now assume $f$ is algebraic over $A$. So $y$ is algebraic if and only if $y-f/\prod_{i=1}^d (1-x_i)$ is so, and then by \Cref{Thm: algebraic_eqv_automatic}, if and only if $\{a_u'-f\}_{u\in \bN^d}$ is $q$-automatic, or equivalently, if and only if $\{a_u'\}_{u\in\bN^d}$ is $q$-automatic. This example suggests that the conclusion of Theorem 1.2 may hold under assumptions weaker than requiring \(E\) to be a vector space. \\ 

Consider the power series ring $\bF_2[[x_1,\dots, x_d]]$, and power series $f_0, f_1\in \bF_2[[x_1,\dots, x_d]]$. Suppose the constant term of $f_i$ is $i\in\bF_2$ ($i=0,1$). Then $\{f_0, f_1\}$ is a set of representatives of $\bF_2[[x_1,\dots, x_d]]/(x_1,\dots, x_d)$. Every $y\in \bF_2[[x_1,\dots, x_d]]$ can be uniquely written as 
$$
y=\sum_{u\in \bN^d}a_ux^u, \quad a_u\in\{f_0, f_1\}.
$$ 
Denote by $a_u(0)\in \{0,1\}$ the constant term of $a_u$. Then $a_u=(1-a_u(0))f_0+a_u(0)f_1$. Therefore 
$$y=\sum_{u\in\bN^d}a_ux^u=\sum_{u\in \bN^d} ((1-a_u(0)) f_0+a_u(0) f_1)x^u=f_0\sum_{u\in\bN^d} x^u +(f_1-f_0)\sum_{u\in \bN^d} a_u(0)x^u.$$ 
Assume that $f_0, f_1$ are algebraic over $\bF_2[x_1, \dots, x_d]$. Then 
\begin{center}
    $y$ is algebraic over $\bF_2[x_1,\dots, x_d]$ $\Leftrightarrow$ $\{a_u(0)\}_{u\in \bN^d}$ is $2$-automatic $\Leftrightarrow$ $\{a_u\}_{u\in \bN^d}$ is $2$-automatic.
\end{center}
Notice that $\{f_0, f_1\}$ is not necessarily a subgroup of $\bF_2[[x_1, \dots, x_d]]$. 

In the second example, the only restriction on the elements of $E$ is that their are algebraic. These examples suggest one to consider whether \Cref{Thm: algebraic_eqv_automatic} is valid without assuming $E$ a group. However the following example shows that algebraicity of the elements of $E$ alone is insufficient in general: without further structure, the conclusion can fail. \\

Let $A=\bF_3[x,y]$, $\fa=(x,y)$, and $\bx=(x,y)$. Then
$\hat A=\bF_3[[x,y]]$ and $\depth(\fa,A)=\dim A=2$. Take
\[
E=\{e_0,e_1,e_2\}=\{0,\,1-x-y,\,2\}.
\]
The reductions of these elements modulo $(x,y)$ are $0,1,2$, so
$E$ is a complete set of representatives; all its elements even
belong to $A$. We show that the algebraic element $1\in A$ has a
non-$3$-automatic $E$-digit array.

For $m,n\geq 0$, put
\[
b_{mn}=\binom{m+n}{m}\bmod 2\in\{0,1\}\subset\bF_3,
\qquad
B(x,y)=\sum_{m,n\geq0}b_{mn}x^my^n,
\]
and set $b_{mn}=0$ if $m<0$ or $n<0$. The reduction modulo $2$
in the definition of $b_{mn}$ is taken among integers; the resulting
digits $0,1$ are then viewed as elements of $\bF_3$. Pascal's
identity implies that, away from $(0,0)$,
\[
b_{mn}\equiv b_{m-1,n}+b_{m,n-1}\pmod 2.
\]
Write
\[
C(x,y)=1+(x+y)B(x,y)=\sum_{m,n\geq0}c_{mn}x^my^n
\quad\text{in }\bF_3[[x,y]].
\]
Then $c_{00}=1$ and, for $m+n>0$,
$c_{mn}=b_{m-1,n}+b_{m,n-1}\in\{0,1,2\}$.
Thus $c_{mn}=1$ if and only if $b_{mn}=1$, including at the origin.
Since $e_a=a-(x+y)\mathbf{1}_{\{a=1\}}$ for $a\in\bF_3$, we obtain
\[
\sum_{m,n\geq0}e_{c_{mn}}x^my^n
=C(x,y)-(x+y)\sum_{c_{mn}=1}x^my^n
=C(x,y)-(x+y)B(x,y)=1.
\]
By uniqueness of the expansion with respect to $\bx$ and $E$,
the $E$-digits of $1$ are precisely $e_{c_{mn}}$.

For completeness, the binary support of $b_{mn}$ satisfies
\[
b_{mn}=1
\quad\Longleftrightarrow\quad
\text{$m$ and $n$ have no common $1$-position in binary}.
\]
Indeed, in $\bF_2[[X,Y]]$ the identity
$(1-X-Y)^{-1}=\prod_{j\geq0}(1+X^{2^j}+Y^{2^j})$
shows that each binary position may be assigned to $X$, to $Y$,
or to neither, but not to both. Symmetry therefore gives, for $n\geq1$,
\[
c_{nn}=2b_{n-1,n}
=\begin{cases}
2,&n=2^k\text{ for some }k\geq0,\\
0,&\text{otherwise}.
\end{cases}
\]
Suppose the two-dimensional array $c=(c_{mn})$ is $3$-automatic. Then its diagonal 
$d_n=c_{nn}$ would be $3$-automatic as well. Indeed, every sequence 
$(d_{3^e n+r})_{n\geq0}$ in the $3$-kernel of $d$ is the diagonal of 
the array $(c_{3^e m+r,\,3^e n+r})_{m,n\geq0}$ in the $3$-kernel of $c$. 
Thus the $3$-kernel of $d$ is finite. On the other hand, it is easy to see the 
$2$-kernel of $(c_{nn})$ is finite, so it is $2$-automatic. In addition, $c_{nn}$ 
is not ultimately periodic. This contradicts Cobham’s theorem, which says that a 
sequence automatic in both multiplicatively independent bases \(2\) and \(3\) 
must be ultimately periodic. Hence $c_{mn}$ is not $3$-automatic. Finally, 
\(a\mapsto e_a\) is a bijection of finite alphabets, so \((e_{c_{mn}})\) 
is not \(3\)-automatic either.

\subsection{\texorpdfstring{Application to $\beta$-expansions}{Applications}}
Let $\bF_q((1/z))$ be the field of one variable formal Laurent series of the form 
$$y=\sum_{n\ge -m}\frac{a_n}{z^n}, \quad a_n\in\bF_q$$ 
where $m\in\bZ$ is called the degree of $y$. We say a formal Laurent series is algebraic if it is algebraic over $\bF_q(1/z)$. The set of Laurent series of degree at most zero is $\bF_q[[1/z]]$. The integer part $[y]$ and the fraction part $\{y\}$ of $y$ are respectively defined as 
$$
[y]\coloneqq \sum_{-m\le n\le 0}\frac{a_n}{z^n}\quad  \text{ and }\quad \{y\}\coloneqq \sum_{n>0}\frac{a_n}{z^n}=y-[y].
$$ 
Let $\beta$ be a series in $\bF_q((1/z))$ of positive degree. Define 
$$
T(y)\coloneqq \beta y-[\beta y]=\{\beta y\}, \quad \forall y\in \bF_q[[1/z]].
$$ 
Denote by $T^n$ the $n$-th iteration of $T$ for positive integers $n$, and set $T^0(y)=y$. Then for all $y\in\bF_q[[1/z]]$, we have the following \emph{$\beta$-expansion} 
$$y=\frac{a_1}{\beta}+\frac{a_2}{\beta^2}+\cdots,$$ 
where $a_n=[\beta T^{n-1} y]$ are elements of $\bF_q[z]$.  For a $\beta$-expansion of $y$,  denote $d_{\beta}(y)$ the sequence $\{a_n\}_{n\ge 1}$. 
Many properties about $\beta$-expansions have been studied in literature. We concentrate on those related to $q$-automatic sequences. Hbaib and Mkaouar \cite{Hbaib-Mkaouar}*{Theorem 5.4} proved that if $d_{\beta}(1)$ is $q$-automatic, then $\beta$ is algebraic. They also proved the converse statement provided that $\deg \beta=1$. Scheicher and Sirvent \cite[Theorem 6.1]{Scheicher-Sirvent} showed that if $\beta$ is isolated of positive degree, then $d_{\beta}(1)$ is $q$-automatic. They \cite[Theorem 6.6]{Scheicher-Sirvent} also showed that if $\beta$ is a Pisot or Salem series, then $d_{\beta}(\alpha)$ is $q$-automatic if and only if $\alpha$ is algebraic. Both \cite{Hbaib-Mkaouar} and  \cite{Scheicher-Sirvent} left questions concerning automaticity of $\beta$-expansions. We collect those questions below. 

\begin{enumerate}
    \item \cite{Scheicher-Sirvent}*{Remark 6.7}: Is either the converse of \cite{Hbaib-Mkaouar}*{Theorem 5.4} or the converse of \cite[Theorem 6.1]{Scheicher-Sirvent} true. More precisely, whether one of the following statements is true: (i) $d_{\beta}(1)$ is $q$-automatic if and only if $\beta$ is algebraic; and (ii) $d_{\beta}(1)$ is $q$-automatic if and only if $\beta$ is isolated. 
    
    \item \cite{Hbaib-Mkaouar}: For any algebraic $\beta$ of positive degree, is it true that $d_{\beta}(x)$ is $q$-automatic if and only if $x$ is algebraic. 

    \item \cite[Remark 6.7]{Scheicher-Sirvent}: Is the converse of \cite[Theorem 6.6]{Scheicher-Sirvent} true. That is, for any $\beta$ of positive degree, if $d_{\beta}(x)$ is $q$-automatic if and only if $x$ is algebraic, then is $\beta$ a Pisot or Salem series.
\end{enumerate}
By applying \Cref{Thm: algebraic_eqv_automatic}, we prove the following theorem which refines results in both \cite{Hbaib-Mkaouar} and \cite{Scheicher-Sirvent}, and answers above questions.  

\begin{Thm}\label{Thm: beta expansions}
Let $\beta\in\bF_q((1/z))$ be a formal Laurent series of positive degree. Then the following statements are equivalent: 
\begin{enumerate}
    \item $\beta$ is algebraic;
    \item $d_{\beta}(1)$ is $q$-automatic; 
    \item for any $y\in \bF_q[[1/z]]$, $d_{\beta}(y)$ is $q$-automatic if and only if $y$ is algebraic.
\end{enumerate}
\end{Thm}

\begin{proof}
(2) $\Rightarrow$ (1) is a result of Hbaib and Mkaouar \cite{Hbaib-Mkaouar}. (3) $\Rightarrow$ (2) is obvious. Hence it remains to show (1) $\Rightarrow$ (3). 

Suppose $\beta$ is algebraic. It is easy to see $d_{\beta}(y)$ is $q$-automatic if and only if $d_{\beta}(\beta y-a_1)$ is $q$-automatic. Notice that $\beta y-a_1=\beta y-[\beta y]\in \frac{1}{z}\bF_q[[1/z]]$. So it is enough to show (1) $\Rightarrow$ (3) for all $y\in\frac{1}{z}\bF_q[[1/z]]$. Now by replacing $y$ with $T(y)$ if necessary,  we can assume $y\in\frac{1}{z}\bF_q[[1/z]]$. Suppose $y=\frac{a_1}{\beta}+\frac{a_2}{\beta^2}+\cdots$ to be the $\beta$-expansion of $y$. Write $d\coloneqq \deg \beta$, then 
$$\frac{\beta y}{z^{d-1}}=\sum_{i=1}^{\infty}\frac{a_i}{z^{d-1}} \frac{1}{\beta^{i-1}}.$$ 
Notice that for any element $y'\in \frac{1}{z}\bF_q[[1/z]]$ one has $\deg [\beta y']< \deg \beta$, also notice that the fraction part of any formal Laurent series is contained in $ \frac{1}{z}\bF_q[[1/z]]$. Thus one can check that $\deg a_i< \deg \beta$ inductively. So for every $i$, $\frac{a_i}{z^{d-1}}$ is contained in the $\bF_q$-subspace 
$\bigoplus_{j=0}^{d-1} \frac{1}{z^j}\bF_q.$

Take $A=\bF_q[1/z]$. Notice that $\bx=\{\beta\}$ is a regular sequence in the ideal $\fa=\frac{1}{z}\bF_q[[1/z]]$. Let $E=\bigoplus_{j=0}^{d-1} \frac{1}{z^j}\bF_q$.  Then 
$$\frac{\beta y}{z^{d-1}}=\sum_{i=1}^{\infty}\frac{a_i}{z^{d-1}} \frac{1}{\beta^{i-1}}$$
is the expansion of $\frac{\beta y}{z^{d-1}}$ with respect to $E$ and $\bx$. Hence $y$ is algebraic if and only if $\frac{\beta y}{z^{d-1}}$ is algebraic, and the latter holds if and only if $\{\frac{a_i}{z^{d-1}}\}_{i\ge 1}$ is $q$-automatic  according to \Cref{Thm: algebraic_eqv_automatic}. Finally, it is obvious that  $\{\frac{a_i}{z^{d-1}}\}_{i\ge 1}$ is $q$-automatic exactly when $d_{\beta}(y)$ is $q$-automatic. This proves (1) $\Rightarrow$ (3) and completes the proof of this theorem.
\end{proof}

\section*{Acknowledgements}
The first author is supported by the National Natural Science Foundation of China (Grant No. 12271382). We would like to thank Qixiao Ma for suggesting the method of henselization. We are also grateful to Shaoshi Chen and Yining Hu for inspiring discussions.

\bibliography{newbib}

@book {AS-auto-seqn,
    AUTHOR = {Allouche, Jean-Paul and Shallit, Jeffrey},
     TITLE = {Automatic sequences},
      NOTE = {Theory, applications, generalizations},
 PUBLISHER = {Cambridge University Press, Cambridge},
      YEAR = {2003},
     PAGES = {xvi+571},
      ISBN = {0-521-82332-3},
   MRCLASS = {11B85 (11Z05 37A45 37B10 68Q45 68R15 94A45)},
  MRNUMBER = {1997038},
MRREVIEWER = {Val\'{e}rie\ Berth\'{e}},
       DOI = {10.1017/CBO9780511546563},
       URL = {https://doi.org/10.1017/CBO9780511546563},
}

@book {Hartshorne,
    AUTHOR = {Hartshorne, Robin},
     TITLE = {Algebraic geometry},
    SERIES = {Graduate Texts in Mathematics, No. 52},
 PUBLISHER = {Springer-Verlag, New York-Heidelberg},
      YEAR = {1977},
     PAGES = {xvi+496},
      ISBN = {0-387-90244-9},
   MRCLASS = {14-01},
  MRNUMBER = {463157},
MRREVIEWER = {Robert Speiser},
}

@misc{stacks-project,
  author       = {The {Stacks project authors}},
  title        = {The Stacks project},
  howpublished = {\url{https://stacks.math.columbia.edu}},
  year         = {2025},
  note         ={\url{https://stacks.math.columbia.edu}},
}

@book {Cohen-Macaulay-rings,
    AUTHOR = {Bruns, Winfried and Herzog, J\"{u}rgen},
     TITLE = {Cohen-{M}acaulay rings},
    SERIES = {Cambridge Studies in Advanced Mathematics},
    VOLUME = {39},
 PUBLISHER = {Cambridge University Press, Cambridge},
      YEAR = {1993},
     PAGES = {xii+403},
      ISBN = {0-521-41068-1},
   MRCLASS = {13H10 (13-02)},
  MRNUMBER = {1251956},
MRREVIEWER = {Matthew Miller},
}

@article{Salon87,
    AUTHOR = {Salon, Olivier},
     TITLE = {Suites automatiques \`a{} multi-indices et alg\'ebricit\'e},
   JOURNAL = {C. R. Acad. Sci. Paris S\'er. I Math.},
  FJOURNAL = {Comptes Rendus des S\'eances de l'Acad\'emie des Sciences.
              S\'erie I. Math\'ematique},
    VOLUME = {305},
      YEAR = {1987},
    NUMBER = {12},
     PAGES = {501--504},
      ISSN = {0249-6291},
   MRCLASS = {11T99 (11K31 13F25)},
  MRNUMBER = {916320},
}

@misc{Salon89,
	author = {Salon, Oliver},
	title = {Propri\'et\'es arithm\'etiques des automates multidimensionnels},
	howpublished = {PhD thesis, Universit\'e Bordeaux I},
	year = {1989},
    note = {PhD Thesis, Universit\'e Bordeaux I}
}

@article{Sharif-Woodcock,
    AUTHOR = {Sharif, Habib and Woodcock, Christopher F.},
     TITLE = {Algebraic functions over a field of positive characteristic
              and {H}adamard products},
   JOURNAL = {J. London Math. Soc. (2)},
  FJOURNAL = {Journal of the London Mathematical Society. Second Series},
    VOLUME = {37},
      YEAR = {1988},
    NUMBER = {3},
     PAGES = {395--403},
      ISSN = {0024-6107,1469-7750},
   MRCLASS = {12E99},
  MRNUMBER = {939116},
MRREVIEWER = {N.\ Sankaran},
       DOI = {10.1112/jlms/s2-37.3.395},
       URL = {https://doi.org/10.1112/jlms/s2-37.3.395},
}

@article{Harase88,
    AUTHOR = {Harase, Takashi},
     TITLE = {Algebraic elements in formal power series rings},
   JOURNAL = {Israel J. Math.},
  FJOURNAL = {Israel Journal of Mathematics},
    VOLUME = {63},
      YEAR = {1988},
    NUMBER = {3},
     PAGES = {281--288},
      ISSN = {0021-2172},
   MRCLASS = {13F25 (11T99)},
  MRNUMBER = {969943},
MRREVIEWER = {Arkadiusz\ P\l oski},
       DOI = {10.1007/BF02778035},
       URL = {https://doi.org/10.1007/BF02778035},
}

@article{Harase89,
    AUTHOR = {Harase, Takashi},
     TITLE = {Algebraic elements in formal power series rings. {II}},
   JOURNAL = {Israel J. Math.},
  FJOURNAL = {Israel Journal of Mathematics},
    VOLUME = {67},
      YEAR = {1989},
    NUMBER = {1},
     PAGES = {62--66},
      ISSN = {0021-2172},
   MRCLASS = {13F25 (11T99)},
  MRNUMBER = {1021361},
MRREVIEWER = {Arkadiusz\ P\l oski},
       DOI = {10.1007/BF02764899},
       URL = {https://doi.org/10.1007/BF02764899},
}

@article{Kedelaya,
    AUTHOR = {Kedlaya, Kiran S.},
     TITLE = {Finite automata and algebraic extensions of function fields},
   JOURNAL = {J. Th\'eor. Nombres Bordeaux},
  FJOURNAL = {Journal de Th\'eorie des Nombres de Bordeaux},
    VOLUME = {18},
      YEAR = {2006},
    NUMBER = {2},
     PAGES = {379--420},
      ISSN = {1246-7405,2118-8572},
   MRCLASS = {11B85 (11D88 13F25 16W60 68Q45 68Q70)},
  MRNUMBER = {2289431},
MRREVIEWER = {Jeffrey\ O.\ Shallit},
       DOI = {10.5802/jtnb.551},
       URL = {https://doi.org/10.5802/jtnb.551},
}

@article {Hbaib-Mkaouar,
    AUTHOR = {Hbaib and Mkaouar},
     TITLE = {Sur le b\^eta-d\'eveloppement de 1 dans le corps des s\'eries
              formelles},
   JOURNAL = {Int. J. Number Theory},
  FJOURNAL = {International Journal of Number Theory},
    VOLUME = {2},
      YEAR = {2006},
    NUMBER = {3},
     PAGES = {365--378},
      ISSN = {1793-0421,1793-7310},
   MRCLASS = {11A67 (11B85 11Y55)},
  MRNUMBER = {2264597},
MRREVIEWER = {Peter\ J.\ Grabner},
       DOI = {10.1142/S1793042106000619},
       URL = {https://doi.org/10.1142/S1793042106000619},
}

@article{Scheicher-Sirvent,
    AUTHOR = {Scheicher, Klaus and Sirvent, V\'ictor F.},
     TITLE = {Automatic {$\beta$}-expansions of formal {L}aurent series over
              finite fields},
   JOURNAL = {Finite Fields Appl.},
  FJOURNAL = {Finite Fields and their Applications},
    VOLUME = {27},
      YEAR = {2014},
     PAGES = {1--23},
      ISSN = {1071-5797,1090-2465},
   MRCLASS = {11K16 (11A63 12E20 37B10)},
  MRNUMBER = {3168980},
MRREVIEWER = {Manfred\ G.\ Madritsch},
       DOI = {10.1016/j.ffa.2013.12.005},
       URL = {https://doi.org/10.1016/j.ffa.2013.12.005},
}

@article{Denef-Lipshitz,
    AUTHOR = {Denef and Lipshitz},
     TITLE = {Algebraic power series and diagonals},
   JOURNAL = {J. Number Theory},
  FJOURNAL = {Journal of Number Theory},
    VOLUME = {26},
      YEAR = {1987},
    NUMBER = {1},
     PAGES = {46--67},
      ISSN = {0022-314X,1096-1658},
   MRCLASS = {16A05},
  MRNUMBER = {883533},
MRREVIEWER = {M.\ Mend\`es France},
       DOI = {10.1016/0022-314X(87)90095-3},
       URL = {https://doi.org/10.1016/0022-314X(87)90095-3},
}

@article{Uramoto,
    AUTHOR = {Uramoto, Takeo},
     TITLE = {Semi-{G}alois categories {II}: {A}n arithmetic analogue of
              {C}hristol's theorem},
   JOURNAL = {J. Algebra},
  FJOURNAL = {Journal of Algebra},
    VOLUME = {508},
      YEAR = {2018},
     PAGES = {539--568},
      ISSN = {0021-8693},
   MRCLASS = {11B85 (13F35 68Q70)},
  MRNUMBER = {3810306},
MRREVIEWER = {Vilius Stakenas},
       DOI = {10.1016/j.jalgebra.2018.04.033},
       URL = {https://doi.org/10.1016/j.jalgebra.2018.04.033},
}

@book{Eisenbud,
    AUTHOR = {Eisenbud, David},
     TITLE = {Commutative algebra},
    SERIES = {Graduate Texts in Mathematics},
    VOLUME = {150},
      NOTE = {With a view toward algebraic geometry},
 PUBLISHER = {Springer-Verlag, New York},
      YEAR = {1995},
     PAGES = {xvi+785},
      ISBN = {0-387-94268-8; 0-387-94269-6},
   MRCLASS = {13-01 (14A05)},
  MRNUMBER = {1322960},
MRREVIEWER = {Matthew\ Miller},
       DOI = {10.1007/978-1-4612-5350-1},
       URL = {https://doi.org/10.1007/978-1-4612-5350-1},
}

@book {Raynaud,
    AUTHOR = {Raynaud, Michel},
     TITLE = {Anneaux locaux hens\'{e}liens},
    SERIES = {Lecture Notes in Mathematics, Vol. 169},
 PUBLISHER = {Springer-Verlag, Berlin-New York},
      YEAR = {1970},
     PAGES = {v+129},
   MRCLASS = {13.95},
  MRNUMBER = {277519},
MRREVIEWER = {J.-P. Lafon},
}

@article {Artin-approximation,
    AUTHOR = {Artin, M.},
     TITLE = {Algebraic approximation of structures over complete local
              rings},
   JOURNAL = {Inst. Hautes \'{E}tudes Sci. Publ. Math.},
  FJOURNAL = {Institut des Hautes \'{E}tudes Scientifiques. Publications
              Math\'{e}matiques},
    NUMBER = {36},
      YEAR = {1969},
     PAGES = {23--58},
      ISSN = {0073-8301},
   MRCLASS = {14.05},
  MRNUMBER = {268188},
MRREVIEWER = {H. Kurke},
       URL = {http://www.numdam.org/item?id=PMIHES_1969__36__23_0},
}

@book {Matsumura-commutative-ring-theory-1986,
    AUTHOR = {Matsumura, Hideyuki},
     TITLE = {Commutative ring theory},
    SERIES = {Cambridge Studies in Advanced Mathematics},
    VOLUME = {8},
      NOTE = {Translated from the Japanese by M. Reid},
 PUBLISHER = {Cambridge University Press, Cambridge},
      YEAR = {1986},
     PAGES = {xiv+320},
      ISBN = {0-521-25916-9},
   MRCLASS = {13-01},
  MRNUMBER = {879273},
MRREVIEWER = {W. V. Vasconcelos},
}

@book {MilEC,
    AUTHOR = {Milne, James S.},
     TITLE = {\'{E}tale cohomology},
    SERIES = {Princeton Mathematical Series},
    VOLUME = {33},
 PUBLISHER = {Princeton University Press, Princeton, N.J.},
      YEAR = {1980},
     PAGES = {xiii+323},
      ISBN = {0-691-08238-3},
   MRCLASS = {14-02 (14F20 18F99)},
  MRNUMBER = {559531},
MRREVIEWER = {G. Horrocks},
}

@book {Matsumura-commutative-alg,
    AUTHOR = {Matsumura, Hideyuki},
     TITLE = {Commutative algebra},
    SERIES = {Mathematics Lecture Note Series},
    VOLUME = {56},
   EDITION = {Second},
 PUBLISHER = {Benjamin/Cummings Publishing Co., Inc., Reading, Mass.},
      YEAR = {1980},
     PAGES = {xv+313},
      ISBN = {0-8053-7026-9},
   MRCLASS = {13-02},
  MRNUMBER = {575344},
MRREVIEWER = {R. M. Fossum},
}

@article {Greco,
    AUTHOR = {Greco, Silvio},
     TITLE = {Henselization of a ring with respect to an ideal},
   JOURNAL = {Trans. Amer. Math. Soc.},
  FJOURNAL = {Transactions of the American Mathematical Society},
    VOLUME = {144},
      YEAR = {1969},
     PAGES = {43--65},
      ISSN = {0002-9947,1088-6850},
   MRCLASS = {13.95},
  MRNUMBER = {251030},
MRREVIEWER = {W.\ Heinzer},
       DOI = {10.2307/1995268},
       URL = {https://doi.org/10.2307/1995268},
}

@book {Atiyah-Macdonald,
    AUTHOR = {Atiyah, M. F. and Macdonald, I. G.},
     TITLE = {Introduction to commutative algebra},
    SERIES = {Addison-Wesley Series in Mathematics},
   EDITION = {economy},
      NOTE = {For the 1969 original see [ MR0242802]},
 PUBLISHER = {Westview Press, Boulder, CO},
      YEAR = {2016},
     PAGES = {ix+128},
      ISBN = {978-0-8133-5018-9; 0-201-00361-9; 0-201-40751-5},
   MRCLASS = {13-01 (13-00)},
  MRNUMBER = {3525784},
}

@book{MacLane,
    author = {Saunders Mac Lane},
    title = {Categories for the Working Mathematician},
    publisher = {Springer New York, NY},
    year = {2013}
}

@article {Adamczewski-Bugeaud,
    AUTHOR = {Adamczewski, Boris and Bugeaud, Yann },
     TITLE = {On the complexity of algebraic numbers I. Expansions in integer base},
   JOURNAL = {Annals of Mathematics},
    VOLUME = {165},
      YEAR = {2007},
     PAGES = {547--565},
}

@article {Borel,
    AUTHOR = {Borel, \'Emile},
     TITLE = {Sur les chiffres d\'ecimaux de {$\sqrt{2}$} et divers
              probl\`emes de probabilit\'es en cha\^ine},
   JOURNAL = {C. R. Acad. Sci. Paris},
  FJOURNAL = {Comptes Rendus Hebdomadaires des S\'eances de l'Acad\'emie des
              Sciences},
    VOLUME = {230},
      YEAR = {1950},
     PAGES = {591--593},
      ISSN = {0001-4036},
   MRCLASS = {60.0X},
  MRNUMBER = {34544},
MRREVIEWER = {W.\ Feller},
}

@article {Allouche,
    AUTHOR = {Allouche, Jean-Paul },
     TITLE = {Sur la transcendance de la serie formelle $\Pi$},
   JOURNAL = {Seminaire de theorie des nombres de Bordeaux},
    VOLUME = {2},
      YEAR = {1990},
     PAGES = {103-117},
}
	
\ContactInfo

\end{document}